\documentclass{article}
\usepackage{amssymb}
\usepackage{latexsym}
\usepackage{amsmath}
\usepackage{mathrsfs}
\usepackage{color}
\usepackage{CJK}
\usepackage{graphicx}

\newtheorem{theorem}{Theorem}[section]

\newtheorem{lemma}{Lemma}[section]
\newtheorem{proposition}{Proposition}[section]
\newtheorem{remark}{Remark}[section]

\numberwithin{equation}{section}
\newenvironment{proof}{\medskip\par\noindent{\bf Proof.}\ }{\qquad
\raisebox{-0.5mm}{\rule{1.5mm}{4mm}}\vspace{6pt}}

\newcommand{\bbr}{\mathbb{R}}

\newcommand{\h}{H^1_0(\Omega)}

\newcommand{\bbn}{\mathbb{N}}
\newcommand{\ve}{\varepsilon}

\begin{document}
\title
{\Large\bf On a critical Kirchhoff problem in high dimensions}%

\author{
Yisheng Huang$^{a},$\thanks{E-mail address: yishengh@suda.edu.cn(Yisheng Huang)}\quad
Zeng Liu$^{b},$\thanks{E-mail address: luckliuz@163.com(Zeng Liu)}\quad
Yuanze Wu$^{c}$\thanks{Corresponding
author. E-mail address: wuyz850306@cumt.edu.cn (Yuanze Wu).}\\%
\footnotesize$^{a}${\em  Department of Mathematics, Soochow University, Suzhou 215006, P.R. China }\\%
\footnotesize$^{b}${\em  Department of Mathematics, Suzhou University of Science and Technology, Suzhou 215009, P.R. China}\\
\footnotesize$^{c}${\em  College of Sciences, China University of Mining and Technology, Xuzhou 221116, P.R. China }}%
\date{}
\maketitle

\noindent{\bf Abstract:} In this paper, we consider the following Kirchhoff problem
$$
\left\{\aligned -\bigg(a+b\int_{\Omega}|\nabla u|^2dx\bigg)\Delta u&= \lambda u^{q-1} + \mu u^{2^*-1}, &\quad \text{in }\Omega, \\
u&>0,&\quad\text{in }\Omega,\\
u&=0,&\quad\text{on }\partial\Omega,
\endaligned
\right.\eqno{(\mathcal{P})}
$$
where $\Omega\subset \bbr^N(N\geq4)$ is a bounded domain, $2\leq q<2^*$, $2^*=\frac{2N}{N-2}$ is the critical Sobolev exponent and $a$, $b$, $\lambda$, $\mu$ are positive parameters.  By using the variational method, we obtain some existence and nonexistence results to $(\mathcal{P})$ for all $N\geq4$ with some further conditions on the parameters $a$, $b$, $\lambda$, $\mu$, which partially improve some known results in the literatures.  Furthermore,
Our result for $N=4$ and $q>2$, together with our previous works \cite{HLW15,HLW151}, gives an almost positive answer to Neimen's open question [J. Differential Equations, 257 (2014), 1168--1193].%

\vspace{6mm} \noindent{\bf Keywords:} Kirchhoff type problem; Critical Sobolev exponent; Positive solution; Variational method.%

\vspace{6mm}\noindent {\bf AMS} Subject Classification 2010: 35B09; 35B33; 35J20.%

\section{Introduction}
In this paper, we study the following Kirchhoff problem
$$
\left\{\aligned -\bigg(a+b\int_{\Omega}|\nabla u|^2dx\bigg)\Delta u&= \lambda u^{q-1} + \mu u^{2^*-1}, &\quad \text{in }\Omega, \\
u&>0,&\quad\text{in }\Omega,\\
u&=0,&\quad\text{on }\partial\Omega,
\endaligned
\right.\eqno{(\mathcal{P}_{a,b,\lambda,\mu})}
$$
where $\Omega\subset \bbr^N(N\geq4)$ is a bounded domain, $2\leq q<2^*$, $2^*=\frac{2N}{N-2}$ is the critical Sobolev exponent and $a$, $b$, $\lambda$, $\mu$ are positive parameters.%

The elliptic type Kirchhoff problem (Kirchhoff problem for short) in a domain $\Omega\subset \bbr^N(1\leq N\leq3)$ has a nice background in physics.  Indeed, such problem is related to the stationary analogue of the following model:
\begin{equation}\label{eq001}
\left\{\aligned &u_{tt}-\bigg(a+b\int_{\Omega}|\nabla u|^2dx\bigg)\Delta u=h(x,u)\quad\text{in }\Omega\times(0, T),\\
&u=0\quad\text{on }\partial\Omega\times(0, T),\\
&u(x,0)=u_0(x),\quad u_t(x,0)=u^*(x),\endaligned\right.
\end{equation}
where $T>0$ is a constant, $u_0, u^*$ are continuous functions.  Such model was first proposed by Kirchhoff in 1883 as an extension of the classical D'Alembert's wave equations for free vibration of elastic strings, Kirchhoff's model takes into account the changes in length of the string produced by transverse vibrations.  In \eqref{eq001}, $u$ denotes the displacement, the nonlinearity $h(x,u)$ denotes the external force and the parameter $a$ denotes the initial tension while the parameter $b$ is related to the intrinsic properties of the string (such as Young¡¯s modulus).  For more details on the physical background of the Kirchhoff problem, we refer the readers to \cite{A12,K83}.

Under some suitable assumptions on the nonlinearities, the Kirchhoff problem has a variational structure in a proper Hilbert space.  Thus, it is natural to study the Kirchhoff problem by the variational method.  However, since the Kirchhoff term $-b(\int_{\Omega}|\nabla u|^2dx)\Delta u$ is non-local and $u\mapsto -b(\int_{\Omega}|\nabla u|^2dx)\Delta u$ is not weakly continuous, a typical difficulty of such problem by using the variational method is that the weak limit of the $(PS)$ sequence to the corresponding functional is not the weak solution of the equation in general.  In order to overcome this difficulty, several methods have been developed and various existence and multiplicity results of nontrivial solutions for the Kirchhoff problem in a domain $\Omega\subset \bbr^N(1\leq N\leq3)$ have been established by the variational method in the literatures, see for example \cite{CWL12,F13,HLW151,LLS14,N141,N14,PZ06,W15,ZP06} and the references therein for the bounded $\Omega$ and \cite{AF12,DPS15,G15,HZ12,HLP14,HL15,HLW16,LLS12,LY131,LLT15,SW14,WTXZ12,WHL15} and the references therein for $\Omega=\mathbb{R}^N$.

Recently, the Kirchhoff problem in high dimensions $(N\geq4)$ has begun to attract much attention.  From the view point of calculus of variation, such problem is much more complex and difficult since the order of the Kirchhoff type non-local term $-b(\int_{\Omega}|\nabla u|^2dx)\Delta u$ in the corresponding functional is $4$, which equals to the critical Sobolev exponent $2^*$ for $N=4$ and is greater than $2^*$ for $N\geq5$.  This fact leads to a big difficulty to obtain a compact $(PS)$ sequence for the corresponding functional.  By making some very careful and complex analyses on the (PS) sequence or using the scaling technique, several existence and multiplicity results of nontrivial solutions to the Kirchhoff problem in high dimensions have been established in the literatures, see for example \cite{CKW11,HLW15,HLW151,HLW16,LLT15,N141,N14,WHL15,ZP06}, etc.,  where \cite{HLW15,HLW151,N14} are devoted to the problem $(\mathcal{P}_{a,b,\lambda,\mu})$.  For the sake of clarity, we divide the remaining of introduction into the following two parts according to the dimensions $N=4$ and $N\geq 5$.

\subsection{The case $N=4$}
In \cite{N14}, by establishing a global splitting result of the $(PS)$ sequence to the corresponding functional of $(\mathcal{P}_{a,b,\lambda,\mu})$ with $N=4$ and applying the mountain pass theorem, Naimen proved that $(\mathcal{P}_{a,b,\lambda,\mu})$ has a solution for $N=4$, $q=2$ and $\lambda<a\sigma_1$, where $\sigma_1$ is the first eigenvalue of $-\Delta$ in $L^2(\Omega)$.  In our recent work \cite{HLW15}, by combining
Neimen's splitting result and the linking theorem, we obtain some existence results of solutions to $(\mathcal{P}_{a,b,\lambda,\mu})$ with $N=4$ in both the following two cases:
\begin{itemize}
\item $\lambda<a\sigma_1$ and $\mu>b\mathcal{S}^2$, where $\mathcal{S}>0$ is the usual Sobolev constant defined by%
\begin{equation*}
\mathcal{S}:=\inf\{\|\nabla u\|_{L^2(\Omega)}^2\mid u\in H_0^{1}(\Omega), \|u\|_{L^{2^*}(\Omega)}^2=1\}%
\end{equation*}
and $\|\cdot\|_{L^r(\Omega)}$ is the usual norm in $L^r(\Omega)\, (r\geq1)$.
\item $\lambda>a\sigma_1$ and $\mu<b\mathcal{S}^2$.
\end{itemize}

In our very recent work \cite{HLW151}, by introducing a new scaling technique, we also obtain some special kinds of solutions with precise expressions on the parameters $a$, $b$, $\lambda$ and $\mu$ to the problem  $(\mathcal{P}_{a,b,\lambda,\mu})$ on a ball for $N=4$ and $q=2$ in the above two cases.

It is clear that $(\mathcal{P}_{a,b,\lambda,\mu})$  has a ground state solution if $N=4$, $q=2$, $\lambda>a\sigma_1$ and $\mu<b\mathcal{S}^2$ since now the corresponding functional of $(\mathcal{P}_{a,b,\lambda,\mu})$ is coerce (cf. \cite[Theorem~1.2]{HLW15}).  Here, we say $u$ is a ground state solution to Problem~$(\mathcal{P}_{a,b,\lambda,\mu})$ if $u\in H_0^1(\Omega)$ and  $\mathcal{E}(u)=\inf_{v\in\mathcal{K}}\mathcal{E}(v)$, where $\mathcal{E}:H_0^1(\Omega)\to\mathbb{R}$ is the corresponding $C^2$-functional of $(\mathcal{P}_{a,b,\lambda,\mu})$ given by
\begin{equation}\label{eq01}
\mathcal{E}(u):=\frac{a}2\|\nabla u\|_{L^2(\Omega)}^2+\frac{b}{4}\|\nabla u\|_{L^2(\Omega)}^4-\frac{\lambda}{q}\|u\|_{L^{q}(\Omega)}^{q}-\frac{\mu}{2^*}\|u\|_{L^{2^*}(\Omega)}^{2^*},\quad \forall\, u\in H_0^1(\Omega)
\end{equation}
and $\mathcal{K}:=\{v\in\h\backslash\{0\}\mid \mathcal{E}'(v)=0\}$.
An natural question is {\it whether $(\mathcal{P}_{a,b,\lambda,\mu})$ always has a ground state solution for $N=4$ and $q=2$}?  In this paper, we will explore this question.

When $N=4$ and $2<q<2^*=4$, Naimen also obtained the following existence result to $(\mathcal{P}_{a,b,\lambda,\mu})$ in \cite{N14}:
\begin{theorem}\label{thm0001}  {\em (\cite[Theorem~1.6]{N14})} \
Let $N=4$.  If $b\mathcal{S}^2<\mu<2b\mathcal{S}^2$ and $\Omega\subset \bbr^4$ is strictly star-sharped, then Problem~$(\mathcal{P}_{a,b,\lambda,\mu})$ has a solution under one of the following three cases:
\begin{itemize}
  \item [$(C1)$] $a>0$, $\lambda>0$ is small enough,
  \item [$(C2)$] $\lambda>0$, $a>0$ is large enough,
  \item [$(C3)$] $a>0$, $\lambda>0$ and $\frac{\mu}{b}$ is
  sufficiently close to $\mathcal{S}^2$.
\end{itemize}
\end{theorem}

We note that Theorem~\ref{thm0001} seems not the natural extension of the corresponding results to $(\mathcal{P}_{a,b,\lambda,\mu})$ in cases of $b=0$ and $N=4$ since it is well known that the conditions that $\mu<2b\mathcal{S}^2$, $\Omega$ is strictly star-sharped and $(C1)$--$(C3)$ are not needed for $(\mathcal{P}_{a,0,\lambda,\mu})$ in the case of $N=4$.  Due to this reason, Naimen proposed the following open question in \cite{N14}:
\begin{enumerate}
\item[{\bf(Q)}]\quad {\bf Are the conditions that $\mu<2b\mathcal{S}^2$, $\Omega\subset \bbr^4$ is strictly star-sharped and $(C1)$--$(C3)$ necessary in Theorem~\ref{thm0001}?}
\end{enumerate}

In \cite{HLW15}, by using the variational method and treating the Kirchhoff type non-local term as a perturbation of $(\mathcal{P}_{a,0,\lambda,\mu})$, we give a partial answer to Naimen's open question.  More precisely, we proved the following theorem:
\begin{theorem}\label{thm0002}   {\em (\cite[Theorem~1.4]{HLW15})} \
Let $N=4$. If $b\mathcal{S}^2<\mu$ and one of conditions $(C1)$--$(C3)$ is satisfied, then there exists $b^*>0 $ such that Problem~$(\mathcal{P}_{a,b,\lambda,\mu})$ has a solution for $b\in(0,b^*)$.
\end{theorem}

Since the local problem~$(\mathcal{P}_{a,0,\lambda,\mu})$ is dependent of the parameter $\mu$, $b^*$ given by Theorem~\ref{thm0002} is also dependent of the parameter $\mu$.  Let us denote $b^*$ by $b^*(\mu)$.  Then due to Theorem~\ref{thm0002}, we can see that the conditions in Theorem~\ref{thm0001} that $\mu<2b\mathcal{S}^2$ and $\Omega\subset \bbr^4$ is strictly star-sharped are not needed  for $b<b^*(\mu)$.  Clearly, $b^*(\mu)\leq\mu\mathcal{S}^{-2}$ by Theorem~\ref{thm0002}, and yet the existence of the $b^*(\mu)$ is obtained by assuming the contrary  in \cite{HLW15}, where we neither showed that $b^*(\mu)=\mathcal{S}^{-2}\mu$ nor gave an estimate to $b^*(\mu)$.  Noting that another partial answer to Naimen's open question was given in \cite{HLW151} in which  we obtained the following result:
\begin{theorem}\label{thm0003}  {\em (\cite[Theorem~1.4]{HLW151})} \
Let $N=4$, $a,\lambda>0$ and $\mu>b\mathcal{S}^2$.  If $\Omega=\mathbb{B}_R$ is a ball, then Problem~$(\mathcal{P}_{a,b,\lambda,\mu})$ has a radial solution.
\end{theorem}

It follows from Theorem~\ref{thm0003} that the conditions in Theorem~\ref{thm0001} that $\mu<2b\mathcal{S}^2$ and $(C1)$--$(C3)$ are not necessary if $\Omega$ is a ball.  However, our scaling technique in \cite{HLW151} heavily depends on the fact of $\Omega=\mathbb{B}_R$ and it may be invalid for a general bounded domain $\Omega$.

In present paper, we will also investigate Naimen's open question once more and give an almost positive answer to it, which also partially improves Theorems~\ref{thm0002} and \ref{thm0003}.

Before we state our result to Problem~$(\mathcal{P}_{a,b,\lambda,\mu})$, we will introduce some notations.  Let $\mathcal{E}(u)$ be given in \eqref{eq01}, then for every $u\in\h$, the function $T_u(t):=\mathcal{E}(tu)$ ($t\in\mathbb{R}$) is twice continuously differentiable on $\mathbb{R}$.  Furthermore, $T_u'(t)=0$ if and only if $tu\in\mathcal{N}$, where $\mathcal{N}$ is the Nehari manifold of $\mathcal{E}(u)$ defined by
\begin{eqnarray*}
\mathcal{N}:=\{u\in\h\backslash\{0\}\mid \mathcal{E}'(u)u=0\}.
\end{eqnarray*}
In particular, $T_u'(1)=0$ if and only if $u\in\mathcal{N}$.  Let
\begin{eqnarray*}
\mathcal{N}^-:=\{u\in\mathcal{N}\mid T_u''(1)<0\}.
\end{eqnarray*}
Then our result to Problem~$(\mathcal{P}_{a,b,\lambda,\mu})$ for $N=4$ can be stated as follows.
\begin{theorem}\label{thm0004}
Let $N=4$, $a, b, \lambda>0$ and $\mu>b\mathcal{S}^2$.
\begin{enumerate}
\item[$(a)$] If $q=2$ then for each $\lambda<a\sigma_1$, Problem~$(\mathcal{P}_{a,b,\lambda,\mu})$ has a ground state solution which minimizes the functional $\mathcal{E}(u)$ on $\mathcal{N}$.  Moreover, this ground state solution is also a mountain pass solution.
\item[$(b)$] If $2<q<4$ and
\begin{eqnarray}\label{eq0013}
b<\frac{(q-2)^2(\mu(a\sigma_1)^{\frac{4-q}{q-2}}+\lambda^{\frac{2}{q-2}})}{4\lambda^{\frac{2}{q-2}}+(q-2)^2\mu(a\sigma_1)^{\frac{4-q}{q-2}}}\mu\mathcal{S}^{-2},
\end{eqnarray}
then Problem~$(\mathcal{P}_{a,b,\lambda,\mu})$ has a solution which minimizes the functional $\mathcal{E}(u)$ on $\mathcal{N}^-$.
\end{enumerate}
\end{theorem}

\begin{remark} {\em
\begin{enumerate}
\item[(1)] From $(a)$ of Theorem~\ref{thm0004}, we can see that $(\mathcal{P}_{a,b,\lambda,\mu})$ also has a ground state solution for $N=4$ and $q=2$ with $\lambda<a\sigma_1$ and $\mu>b\mathcal{S}^2$.  On the other hand, it was known that $(\mathcal{P}_{a,b,\lambda,\mu})$ has no solutions for $N=4$ and $q=2$ with $\lambda<a\sigma_1$ and $\mu<b\mathcal{S}^2$ (cf. \cite[Theorem~1.2]{HLW15}), hence the question whether $(\mathcal{P}_{a,b,\lambda,\mu})$ always has a ground state solution for $N=4$ and $q=2$ remains the situation that $\lambda>a\sigma_1$ and $\mu>b\mathcal{S}^2$ to be considered.  However, in this situation, we observe that the functional $\mathcal{E}(u)$ is unbounded from below and indefinite in $\h$,
thus the treatments for the Kirchhoff term $b\|\nabla u\|_{L^2(\Omega)}^4$ must be different from those in this paper and our previous work \cite{HLW15}.
Therefore, how to find out a ground state solution of $(\mathcal{P}_{a,b,\lambda,\mu})$ for $N=4$ and $q=2$ with $\lambda>a\sigma_1$ and $\mu>b\mathcal{S}^2$ is an interesting problem for our future.
\item[(2)]  For the local problem~$(\mathcal{P}_{a,0,\lambda,\mu})$, it is well known that the ground state solution of $(\mathcal{P}_{a,0,\lambda,\mu})$ for $N=4$ and $q=2$ with $\lambda<a\sigma_1$ not only minimizes the functional $\mathcal{E}(u)$ both in $\mathcal{N}$ and $\mathcal{K}$ but also is a mountain pass solution.  Now, by $(a)$ of Theorem~\ref{thm0004}, we can see that the ground state solution of $(\mathcal{P}_{a,b,\lambda,\mu})$ for $N=4$ and $q=2$ with $\lambda<a\sigma_1$ and $\mu>b\mathcal{S}^2$ has the same property.  Due to this fact, $(a)$ of Theorem~\ref{thm0004} can be seen as a complement of Neimen's result in \cite{N14}.
\end{enumerate}}
\end{remark}

\begin{remark}{\em
\begin{enumerate}
\item[(1)] Denote $\tilde{b}(\mu):= \frac{(q-2)^2(\mu(a\sigma_1)^{\frac{4-q}{q-2}}+\lambda^{\frac{2}{q-2}})}{4\lambda^{\frac{2}{q-2}}+(q-2)^2\mu(a\sigma_1)^{\frac{4-q}{q-2}}}$. Clearly $\tilde{b}(\mu)<1$ if $2<q<4$ and it follows from $(b)$ of Theorem~\ref{thm0004} that $b^*(\mu)$ of Theorem~\ref{thm0003} satisfies $b^*(\mu)\geq \tilde{b}(\mu)
    \mu\mathcal{S}^{-2}$, so that $b^*(\mu)$ has an estimate: $\tilde{b}(\mu)\mu\mathcal{S}^{-2}\leq b^*(\mu)\leq \mu\mathcal{S}^{-2}$.  In this sense, $(b)$ of Theorem~\ref{thm0004} partially improves Theorem~\ref{thm0002}.  Also $(b)$ of Theorem~\ref{thm0004}  partially improves Theorem~\ref{thm0003} since the condition $\Omega=\mathbb{B}_R$ in Theorem~\ref{thm0003} is not needed if \eqref{eq0013} holds.
\item[(2)]  Obviously, for each $a,\lambda>0$, it holds that $\tilde{b}(\mu)\to1$ as $\mu\to+\infty$, $\frac{(q-2)^2}{4}<\tilde{b}(\mu)$ and $\tilde{b}(\mu)\to\frac{(q-2)^2}{4}$ as $\mu\to0^+$.  So that we can summarize Theorems~\ref{thm0002}--\ref{thm0003} and $(b)$ of Theorem~\ref{thm0004} into the following figure and table.

\begin{figure}[htbp]
\centerline{\includegraphics[width=8cm]{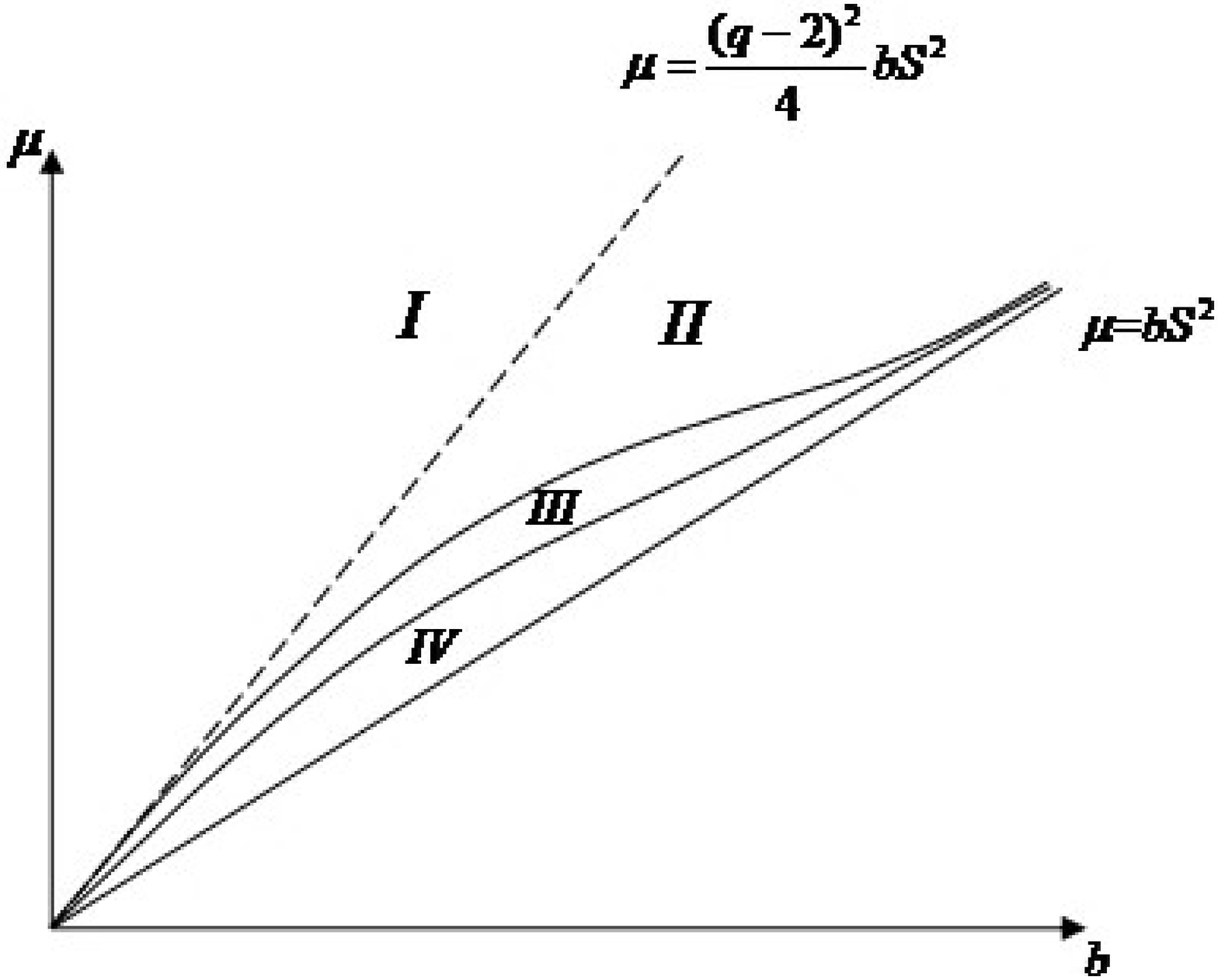}}
\caption{The plane of $(b,\mu)$}\label{Fig01}
\end{figure}

\begin{center}
\begin{tabular}{|l|c|c|}\hline
Ranges of $(b,\mu)$ & Corresponding regions  & The answer of  \\
& in Fig.~\ref{Fig01} & Naimen's open question \\
\cline{1-3} $b<\frac{(q-2)^2}{4}\mu\mathcal{S}^{-2}$ & I & Positive\\
\cline{1-3} $\frac{(q-2)^2}{4}\mu\mathcal{S}^{-2}\leq b<\tilde{b}(\mu)\mu\mathcal{S}^{-2}$ & II & Positive\\
\cline{1-3}& & Partial positive: need \\
$\tilde{b}(\mu)\mu\mathcal{S}^{-2}\leq b<b^*(\mu)$& III & further condition $\Omega=\mathbb{B}_R$ or\\
& & one of $(C1)$--$(C3)$ \\
\cline{1-3} & & Partial positive: need \\
 $b^*(\mu)\leq b<\mu\mathcal{S}^{-2}$ &IV & further condition $\Omega=\mathbb{B}_R$\\
\hline
\end{tabular}
\end{center}
 By the above table, we can see that Naimen's open question has a positive answer in the regions~I and II of Fig.~\ref{Fig01} and has a partial answer in the regions~III and IV of Fig.~\ref{Fig01}.  Thus, it seems that Theorems~\ref{thm0002}--\ref{thm0003} and $(b)$ of Theorem~\ref{thm0004} give an almost positive answer to Naimen's open question.
\end{enumerate}}
\end{remark}

\subsection{The case $N\geq 5$}
To the best of our knowledge, there is few study on $(\mathcal{P}_{a,b,\lambda,\mu})$ for $N\geq5$ in literatures except our very recent work \cite{HLW151}, where, by introducing a new scaling technique, we obtained some existence results to $(\mathcal{P}_{a,b,\lambda,\mu})$ on a ball for $N\geq5$.  As we stated in the above, our scaling technique in \cite{HLW151} heavily dependents the fact that $\Omega=\mathbb{B}_R$ and it may be invalid for a general bounded domain. We also remark that, as observed by Neimen in \cite{N14}, the energy values of bubbles to the $(PS)$ sequence may be negative at some low energy values.  This fact leads that it is very hard to find out a compact $(PS)$ sequence to the corresponding functional $\mathcal{E}(u)$  of $(\mathcal{P}_{a,b,\lambda,\mu})$ on a general bounded domain for $N\geq5$, which seems to be the main difficulty that prevent one to obtain a solution to $(\mathcal{P}_{a,b,\lambda,\mu})$ on a general bounded domain for $N\geq5$ by using the variational method.  In this paper, by finding out a special bounded $(PS)$ sequence of  $\mathcal{E}(u)$  and analyzing carefully the compactness of this $(PS)$ sequence, we give some existence and nonexistence results to $(\mathcal{P}_{a,b,\lambda,\mu})$ for $N\geq5$. Now, let us state our results in the following.
\begin{theorem}\label{thm0005}
Let $N\geq5$.
\begin{enumerate}
\item If $q=2$ and $\lambda<a\sigma_1$ then Problem~$(\mathcal{P}_{a,b,\lambda,\mu})$ has a solution which minimizes the functional $\mathcal{E}(u)$ on $\mathcal{N}^-$ under the following conditions
\begin{enumerate}
\item[$(D0)$] $\frac{N(2^*-2)^2a}{42^*(4-2^*)b}(1-\frac{\lambda}{a\sigma_1})>\mathcal{S}^{\frac N2}\bigg[\frac{8a}{(2^*+2)(4-2^*)\mu}\bigg]^{\frac{2}{2^*-2}}$;
\item[$(D1)$] $a-\frac{(4-2^*)\mu}{2}\bigg[\frac{(2^*-2)\mu}{2b\mathcal{S}^{\frac N2}}\bigg]^{\frac{2^*-2}{4-2^*}}<0$.
\end{enumerate}
Moreover, Problem~$(\mathcal{P}_{a,b,\lambda,\mu})$ has no solution if
\begin{eqnarray*}
a(1-\frac{\lambda}{a\sigma_1})\bigg[\frac{a\mathcal{S}^{\frac{2^*}{2}}}{\mu}(1-\frac{\lambda}{a\sigma_1})\bigg]^{\frac{2}{2^*-2}}>\bigg[\mu\mathcal{S}^{-\frac{2^*}{2}}b^{-\frac{2^*}{4}}\bigg]^{\frac{4}{4-2^*}}.
\end{eqnarray*}
\item If $q>2$ then Problem~$(\mathcal{P}_{a,b,\lambda,\mu})$ has a solution which minimizes the functional $\mathcal{E}(u)$ on $\mathcal{N}^-$ under the following conditions
\begin{enumerate}
\item[$(D1)$] $a-\frac{(4-2^*)\mu}{2}\bigg[\frac{(2^*-2)\mu}{2b\mathcal{S}^{\frac N2}}\bigg]^{\frac{2^*-2}{4-2^*}}<0$;
\item[$(D2)$] $a-\frac{(4-2^*)(2^*+2-q)}{2(4-q)}\bigg[\frac{(2^*-2)(2^*-q)\mu}{2(4-q)b\mathcal{S}^{\frac N2}}\bigg]^{\frac{2^*-2}{4-2^*}}<0$;
\item[$(D3)$] $\frac{4q}{N(q-2)}\mathcal{S}^{\frac N2}\bigg[\frac{8a}{(2^*+2)(4-2^*)\mu}\bigg]^{\frac{2}{2^*-2}}<\frac{(q-2)a}{(4-q)b}$.
\end{enumerate}
Moreover, Problem~$(\mathcal{P}_{a,b,\lambda,\mu})$ has no solution if
\begin{eqnarray*}
b\geq\bigg(\lambda \mathcal{S}_q^{-\frac{q}{2}}(\frac{2}{a})^{4-q}\bigg)^{\frac{2}{q-2}}+\bigg(\mu \mathcal{S}^{-\frac{2^*}{2}}(\frac{2}{a})^{4-2^*}\bigg)^{\frac{2}{2^*-2}},
\end{eqnarray*}
where $
\mathcal{S}_q:=\inf\{\|\nabla u\|_{L^2(\Omega)}^2\mid u\in H_0^{1}(\Omega), \|u\|_{L^{q}(\Omega)}^2=1\}
$.
\end{enumerate}
\end{theorem}

\begin{remark}
\begin{enumerate}
{\em \item[$(1)$]  To the best of our knowledge, Theorem~\ref{thm0005} seems the first result to $(\mathcal{P}_{a,b,\lambda,\mu})$ on a general bounded domain for $N\geq5$.
\item[$(2)$]  As we will see, the main idea of the proof to Theorem~\ref{thm0005} is similar to that of Theorem~\ref{thm0004}.  However, since $2^*<4$ in the cases $N\geq5$, some new ideas and a much more complex calculation are also needed, especially for the compactness of the $(PS)$ sequence.
\item[$(3)$]  Theorem~\ref{thm0005} can be summarized as follows: for fixed $a,\lambda,\mu>0$ and $N\geq5$, $(\mathcal{P}_{a,b,\lambda,\mu})$ has a solution if $b>0$ small enough and $(\mathcal{P}_{a,b,\lambda,\mu})$ has no solution if $b>0$ large enough.  Now, let
    \begin{eqnarray*}
    \widetilde{b}_1=\sup\{b>0\mid (\mathcal{P}_{a,b,\lambda,\mu})\text{ has a solution for }N\geq5\}
    \end{eqnarray*}
    and
    \begin{eqnarray*}
    \widetilde{b}_2=\inf\{b>0\mid (\mathcal{P}_{a,b,\lambda,\mu})\text{ has no solution for }N\geq5\}.
    \end{eqnarray*}
    Then it is clear that $\widetilde{b}_1\leq\widetilde{b}_2$.  However, we are not sure whether it holds that $\widetilde{b}_1=\widetilde{b}_2$ or not.}
\end{enumerate}
\end{remark}

This paper is organized as follows.  In section~2, we will investigate Neimen's open question and prove Theorem~\ref{thm0004}.  In section~3, we will study $(\mathcal{P}_{a,b,\lambda,\mu})$ on a general bounded domain in the cases $N\geq5$ and prove Theorem~\ref{thm0005}.

Through this paper, $o_n(1)$ will always denote the quantities tending towards zero as $n\to\infty$ and $d_i$ will denote the positive constants which may be different.%

\section{The case $N=4$}

\subsection{The Nehari manifold $\mathcal{N}$}
Let us first make some observations on the Nehari manifold $\mathcal{N}$ by the fibering maps $T_u(t)=\mathcal{E}(tu)$.
\begin{lemma}\label{lem0001}
Let $N=4$ and $\chi_u=\{tu\mid t>0\}$.  Then we have the following.
\begin{enumerate}
\item[$(1)$] If $q=2$ and $0<\lambda<a\sigma_1$, then $\chi_u\cap\mathcal{N}=\emptyset$ for $b\|\nabla u\|_{L^2(\Omega)}^4-\mu\|u\|_{L^4(\Omega)}^4\geq0$ and there exists a unique $t_*(u)>0$ such that $t_*(u)u\in\chi_u\cap\mathcal{N}$ for $b\|\nabla u\|_{L^2(\Omega)}^4-\mu\|u\|_{L^4(\Omega)}^4<0$.  Moreover, $\mathcal{N}=\mathcal{N}^-$ and $T_u(1)=\max_{t\geq0}T_u(t)$ for $u\in\mathcal{N}$.
\item[$(2)$] If $q>2$, then $\chi_u\cap\mathcal{N}\not=\emptyset$ if and only if one of the following three cases happens
\begin{enumerate}
\item[$(i)$] $b\|\nabla u\|_{L^2(\Omega)}^4-\mu\|u\|_{L^4(\Omega)}^4\leq0$;
\item[$(ii)$] $b\|\nabla u\|_{L^2(\Omega)}^4-\mu\|u\|_{L^4(\Omega)}^4>0$ and $\mathcal{D}(u)<0$;
\item[$(iii)$] $b\|\nabla u\|_{L^2(\Omega)}^4-\mu\|u\|_{L^4(\Omega)}^4>0$ and $\mathcal{D}(u)=0$,
\end{enumerate}
where
\begin{eqnarray*}\label{eq1102}
\mathcal{D}(u)=a\|\nabla u\|_{L^2(\Omega)}^2-\frac{4-q}{2}\bigg[\frac{(q-2)\lambda\|u\|_{L^q(\Omega)}^q}{2(b\|\nabla u\|_{L^2(\Omega)}^4-\mu\|u\|_{L^4(\Omega)}^4)}\bigg]^{\frac{q-2}{4-q}}\lambda\|u\|_{L^q(\Omega)}^q.
\end{eqnarray*}
Moreover, there exists a unique $0<t_*(u)$ such that $t_*(u)u\in\chi_u\cap\mathcal{N}^-$ if and only if one of $(i)$ and $(ii)$ happens and $T_u(1)=\max_{t\geq0}T_u(t)$ in the case $(i)$ and $T_u(1)=\max_{0\leq t\leq t_0(u)}T_u(t)$ in the case $(ii)$, where
\begin{eqnarray*}
t_0(u)=\bigg[\frac{(q-2)\lambda\|u\|_{L^q(\Omega)}^q}{2(b\|\nabla u\|_{L^2(\Omega)}^4-\mu\|u\|_{L^4(\Omega)}^4)}\bigg]^{\frac{1}{4-q}}.
\end{eqnarray*}
\end{enumerate}
\end{lemma}
\begin{proof}
Since $2^*=4$ for $N=4$, by a direct calculation, we can see that for every $u\in\h\backslash\{0\}$, we have
\begin{eqnarray*}
T_u'(t)=t(a\|\nabla u\|_{L^2(\Omega)}^2-\lambda\|u\|_{L^q(\Omega)}^qt^{q-2}+(b\|\nabla u\|_{L^2(\Omega)}^4-\mu\|u\|_{L^4(\Omega)}^4)t^2).
\end{eqnarray*}
$(1)$\quad Since $q=2$ and $0<\lambda<a\sigma_1$, by the definition of $\sigma_1$, we can see that $T_u'(t)>0$ for all $t>0$ when $b\|\nabla u\|_{L^2(\Omega)}^4-\mu\|u\|_{L^4(\Omega)}^4\geq0$ and there exists a unique $t_*(u)=\bigg[\frac{a\|\nabla u\|_{L^2(\Omega)}^2-\lambda\|u\|_{L^2(\Omega)}^2}{\mu\|u\|_{L^4(\Omega)}^4-b\|\nabla u\|_{L^2(\Omega)}^4}\bigg]^{\frac12}$ such that $T_u'(t_*(u))=0$ for $b\|\nabla u\|_{L^2(\Omega)}^4-\mu\|u\|_{L^4(\Omega)}^4<0$.  Moreover, $T_u'(t)>0$ for $0<t<t_*(u)$ and $T_u'(t)<0$ for $t>t_*(u)$.  It follows from the definitions of $\mathcal{N}$ and $\mathcal{N}^-$ that $\chi_u\cap\mathcal{N}=\emptyset$ for $b\|\nabla u\|_{L^2(\Omega)}^4-\mu\|u\|_{L^4(\Omega)}^4\geq0$ and there exists a unique $t_*(u)>0$ such that $t_*(u)u\in\chi_u\cap\mathcal{N}$ for $b\|\nabla u\|_{L^2(\Omega)}^4-\mu\|u\|_{L^4(\Omega)}^4<0$.  Furthermore, $\mathcal{N}=\mathcal{N}^-$ and $T_u(1)=\max_{t\geq0}T_u(t)$ for $u\in\mathcal{N}$.

$(2)$\quad If $b\|\nabla u\|_{L^2(\Omega)}^4-\mu\|u\|_{L^4(\Omega)}^4\leq0$, then by $2<q<4$, it is easy to see that there exists a unique $t_*(u)>0$ such that $T_u'(t_*(u))=0$.  Moreover, $T_u'(t)>0$ for $t\in(0, t_*(u))$ and $T_u'(t)<0$ for $t\in(t_*(u), +\infty)$.  Let us consider the case of $b\|\nabla u\|_{L^2(\Omega)}^4-\mu\|u\|_{L^4(\Omega)}^4>0$ in what follows.  Set
\begin{eqnarray*}
T_{u,1}(t)=(b\|\nabla u\|_{L^2(\Omega)}^4-\mu\|u\|_{L^4(\Omega)}^4)t^2-\lambda\|u\|_{L^q(\Omega)}^qt^{q-2}.
\end{eqnarray*}
Since $2<q<4$, by a direct calculation, we can see that
\begin{eqnarray*}
T_{u,1}'(t)=t^{q-3}(2(b\|\nabla u\|_{L^2(\Omega)}^4-\mu\|u\|_{L^4(\Omega)}^4)t^{4-q}-(q-2)\lambda\|u\|_{L^q(\Omega)}^q).
\end{eqnarray*}
Clearly, $T_{u,1}(t_0(u))=\min_{t>0}T_{u,1}(t)$ and $T_{u,1}(t)$ is strictly decreasing on $(0, t_0(u))$ and strictly increasing on $(t_0(u), +\infty)$.  Thus, by a direct calculation, we can see that $\min_{t\geq0}\frac{T_u'(t)}{t}=\frac{T_{u}'(t_0(u))}{t_0(u)}=\mathcal{D}(u)$.  It follows that $T_u'(t)>0$ for all $t>0$ in the subcase of $\mathcal{D}(u)>0$, $T_u'(t)\geq0$ for all $t>0$ and $T_u'(t)=0$ if and only if $t=t_0(u)$ in the subcase of $\mathcal{D}(u)=0$ and there exist unique $0<t_*^1(u)<t_0(u)<t_*^2(u)$ such that $T_u'(t_*^1(u))=T_u'(t_*^2(u))=0$ and $T_u'(t)>0$ for $t\in(0, t_*^1(u))$, $T_u'(t)<0$ for $t\in(t_*^1(u), t_*^2(u))$ and $T_u'(t)>0$ for $t\in(t_*^2(u), +\infty)$.  The conclusions follow immediately from the relations between $T_u(t)$ and $\mathcal{N}$ and $\mathcal{N}^-$.
\end{proof}

Let $\{u_n\}\subset\h$ be a minimizing sequence of $\mathcal{S}$ and satisfy $\|u_n\|_{L^4(\Omega)}=1$.  Then we can easy to see that
\begin{eqnarray*}
b\|\nabla u_n\|_{L^2(\Omega)}^4-\mu\|u_n\|_{L^4(\Omega)}^4&=&b\|\nabla u_n\|_{L^2(\Omega)}^4-\mu\\
&=&b\mathcal{S}^2-\mu+o_n(1).
\end{eqnarray*}
If $b\mathcal{S}^2<\mu$, then $b\|\nabla u_n\|_{L^2(\Omega)}^4-\mu\|u_n\|_{L^4(\Omega)}^4<0$ for $n$ large enough.  Thus, by Lemma~\ref{lem0001}, we have the following.
\begin{lemma}\label{lem0002}
Let $N=4$.  If $b\mathcal{S}^2<\mu$ then both $\mathcal{N}$ and $\mathcal{N}^-$ are nonempty sets under one of the following two cases
\begin{enumerate}
\item[$(a)$] $q=2$ and $0<\lambda<a\sigma_1$;
\item[$(b)$] $q>2$.
\end{enumerate}
\end{lemma}

By Lemma~\ref{lem0002}, we can see that $m^-=\inf_{\mathcal{N}^-}\mathcal{E}(u)$ is well defined.  In what follows, let us give some estimates on $m^-$.
\begin{lemma}\label{lem0003}
Let $N=4$ and $b\mathcal{S}^2<\mu$.  Then we have $0<m^-<\frac{a^2\mathcal{S}^2}{4(\mu-b\mathcal{S}^2)}$ under one of the following two cases
\begin{enumerate}
\item[$(a)$] $q=2$ and $0<\lambda<a\sigma_1$;
\item[$(b)$] $q>2$.
\end{enumerate}
\end{lemma}
\begin{proof}
Let us first estimate the low bound of $m^-$ in the case $q=2$.  Suppose $u\in\mathcal{N}^-$, then by the definitions of $\mathcal{N}^-$ and $\sigma_1$,
\begin{eqnarray}
\mathcal{E}(u)&=&\mathcal{E}(u)-\frac{1}{4}\mathcal{E}'(u)u\notag\\
&=&\frac14(a\|\nabla u\|_{L^2(\Omega)}^2-\lambda\|u\|_{L^2(\Omega)}^2)\notag\\
&\geq&\frac{a}{4}(1-\frac{\lambda}{a\sigma_1})\|\nabla u\|_{L^2(\Omega)}^2.\label{eq5001}
\end{eqnarray}
On the other hand, since $u\in\mathcal{N}^-\subset\mathcal{N}$, we can see from the definitions of $\sigma_1$ and $\mathcal{S}$ that
\begin{eqnarray}
a(1-\frac{\lambda}{a\sigma_1})\|\nabla u\|_{L^2(\Omega)}^2&\leq&(a\|\nabla u\|_{L^2(\Omega)}^2-\lambda\|u\|_{L^2(\Omega)}^2)\notag\\
&=&\mu\|u\|_{L^4(\Omega)}^4-b\|\nabla u\|_{L^2(\Omega)}^4\notag\\
&\leq&(\mu\mathcal{S}^{-2}-b)\|\nabla u\|_{L^2(\Omega)}^4,\label{eq5002}
\end{eqnarray}
which, together with $0<\lambda<a\sigma_1$, implies $\|\nabla u\|_{L^2(\Omega)}\geq d_0>0$ for some positive constant $d_0$.  It follows from \eqref{eq5001} and $0<\lambda<a\sigma_1$ once more that $m^-\geq\frac{a}{4}(1-\frac{\lambda}{a\sigma_1})d_0^2$.  We next estimate the low bound of $m^-$ in the case $q>2$.  Suppose $u\in\mathcal{N}^-$, then by the definition of $\mathcal{N}^-$ we can see that
\begin{eqnarray*}
a\|\nabla u\|_{L^2(\Omega)}^2-\lambda\|u\|_{L^q(\Omega)}^q+b\|\nabla u\|_{L^2(\Omega)}^4-\mu\|u\|_{L^4(\Omega)}^4=0
\end{eqnarray*}
and
\begin{eqnarray*}
a\|\nabla u\|_{L^2(\Omega)}^2-(q-1)\lambda\|u\|_{L^q(\Omega)}^q+3(b\|\nabla u\|_{L^2(\Omega)}^4-\mu\|u\|_{L^4(\Omega)}^4)<0.
\end{eqnarray*}
It follows that
\begin{eqnarray*}
(4-q)(b\|\nabla u\|_{L^2(\Omega)}^4-\mu\|u\|_{L^4(\Omega)}^4)<(q-2)a\|\nabla u\|_{L^2(\Omega)}^2,
\end{eqnarray*}
which together with $u\in\mathcal{N}^-\subset\mathcal{N}$, implies
\begin{eqnarray}
\mathcal{E}(u)&=&\mathcal{E}(u)-\frac{1}{q}\mathcal{E}'(u)u\notag\\
&=&\frac{q-2}{2q}a\|\nabla u\|_{L^2(\Omega)}^2-\frac{4-q}{4q}(b\|\nabla u\|_{L^2(\Omega)}^4-\mu\|u\|_{L^4(\Omega)}^4)\notag\\
&>&\frac{q-2}{4q}a\|\nabla u\|_{L^2(\Omega)}^2.\label{eq0001}
\end{eqnarray}
On the other hand, since $u\in\mathcal{N}^-\subset\mathcal{N}$ and $b>0$, we can see from the definitions of $\mathcal{S}_q$ and $\mathcal{S}$ that
\begin{eqnarray}
a\|\nabla u\|_{L^2(\Omega)}^2&\leq&\lambda\|u\|_{L^q(\Omega)}^q+\mu\|u\|_{L^4(\Omega)}^4\notag\\
&\leq&\lambda\mathcal{S}_q^{-\frac{q}{2}}\|\nabla u\|_{L^2(\Omega)}^q+\mu\mathcal{S}^{-2}\|\nabla u\|_{L^2(\Omega)}^4.\label{eq0002}
\end{eqnarray}
Note that $2<q<4$, we must have from \eqref{eq0002} that $\|\nabla u\|_{L^2(\Omega)}\geq d_0>0$, where $d_0$ is a constant.  Now, by \eqref{eq0001}, we can see that $m^-\geq\frac{a(q-2)d_0^2}{4q}>0$.  We final estimate the up bound of $m^-$ for all $2\leq q<4$.  Let $v_\ve$ be the function given by \cite[(9)]{N14}.  Then by \cite[Lemma~2.2]{N14} and a similar argument as used in \cite[Lemma~3.4]{N14}, for all $2\leq q<4$, we can see that
\begin{eqnarray}\label{eq0003}
\sup_{t\geq0}\mathcal{E}(tv_\ve)<\frac{a^2\mathcal{S}^2}{4(\mu-b\mathcal{S}^2)}\quad\text{if }\ve>0\text{ small enough}.
\end{eqnarray}
On the other hand, by \cite[(9)]{N14} and $\mu>b\mathcal{S}^2$, we can see that $b\|\nabla v_\ve\|_{L^2(\Omega)}^4-\mu\|v_\ve\|_{L^4(\Omega)}^4<0$ for $\ve>0$ small enough.  By Lemma~\ref{lem0001}, there exists a unique $t_*(v_\ve)>0$ such that $t_*(v_\ve)v_\ve\in\mathcal{N}^-$ for all $2\leq q<4$ if $\ve>0$ small enough.  It follows from \eqref{eq0003} that
\begin{eqnarray*}
m^-\leq\mathcal{E}(t_*(v_\ve)v_\ve)\leq\sup_{t\geq0}\mathcal{E}(tv_\ve)<\frac{a^2\mathcal{S}^2}{4(\mu-b\mathcal{S}^2)}\quad\text{for }\ve>0\text{ small enough},
\end{eqnarray*}
which completes the proof.
\end{proof}

\subsection{The subcase $q=2$}
Since Lemma~\ref{lem0001} and \eqref{eq5002} hold, by applying the Ekeland's principle in a standard way, we can see that there exists $\{u_n\}\subset\mathcal{N}^-$ such that
\begin{enumerate}
\item[$(a)$] $\mathcal{E}(u_n)=m^-+o_n(1)$;
\item[$(b)$] $\mathcal{E}'(u_n)=o_n(1)$ in $H^{-1}(\Omega)$.
\end{enumerate}
\begin{proposition}\label{prop0001}
Let $N=4$ and $q=2$.  Then $(\mathcal{P}_{a,b,\lambda,\mu})$ has a ground state solution for $0<\lambda<a\sigma_1$ and $b\mathcal{S}^2<\mu$.
\end{proposition}
\begin{proof}
Thanks to \cite[Remark~4.2]{N14}, we may assume $u_n\geq0$.  By \eqref{eq5001} and $(a)$, we can see that $\{u_n\}$ is bounded in $\h$.  Thanks to Lemma~\ref{lem0003} and \cite[Lemma~2.1]{N14}, we have that $u_n=u_0+o_n(1)$ strongly in $\h$.  It follows from the strong maximum principle and the definition of $\mathcal{N}$ that $u_0$ is a ground state solution of $(\mathcal{P}_{a,b,\lambda,\mu})$.  It remains to show that $u_0$ is also a mountain pass solution.  Indeed, define the mountain pass level as follows:
$$
c:=\inf_{\gamma\in\Gamma}\sup_{u\in\gamma([0,1])}\mathcal{E}(u),\qquad \Gamma:=\{\gamma\in C([0,1],H^1_0(\Omega)):\gamma(0)=0, \mathcal{E}(\gamma(1))<0\}.
$$
Then by Neimen's result (cf.\cite{N14}), there exists $v_0\in \mathcal{N}$ such that $\mathcal{E}(v_0)=c$. It follows that $m^-\leq c$.  On the other hand, since $u_0\in\mathcal{N}$ and $\mathcal{E}(u_0)=m^->0$, by Lemma~\ref{lem0001}, we can see that $\mathcal{E}(tu_0)\to -\infty$ as $t\to \infty$ and $\mathcal{E}(u_0)=\max_{t>0}\mathcal{E}(tu_0)$.  Thus, there exists $s_0>2$ such that $\mathcal{E}(s_0u_0)<0$.  Set $\gamma_0(t)=ts_0u_0$, we have $\gamma_0\in \Gamma$.  Therefore,
$$
c\leq\max_{t\in[0,1]}\mathcal{E}(\gamma_0(t))=\mathcal{E}(u_0)=m^-,
$$
which implies that $u_0$ is also a mountain pass solution.
\end{proof}
\subsection{The subcase $q>2$}
Since $\mathcal{N}^-\not=\emptyset$, by the Ekeland's principle, there exists $\{u_n\}\subset\mathcal{N}^-$ such that
\begin{enumerate}
\item[$(a)$] $\mathcal{E}(u_n)=m^-+o_n(1)$;
\item[$(b)$] $\mathcal{E}(v)-\mathcal{E}(u_n)\geq-\frac1n\|\nabla (v-u_n)\|_{L^2(\Omega)}$ for all $v\in\mathcal{N}^-$.
\end{enumerate}
In what follows, we will show that under some further conditions on the parameters $a,b,\lambda,\mu$, we actually have that $\{u_n\}\subset\mathcal{N}^-$ is a bounded $(PS)$ sequence of $\mathcal{E}(u)$ at the energy value $m^-$ for $2<q<4$.
\begin{lemma}\label{lem0004}
Let $N=4$ and $2<q<4$.  If
\begin{eqnarray*}
b<\frac{(q-2)^2(\mu(a\sigma_1)^{\frac{4-q}{q-2}}+\lambda^{\frac{2}{q-2}})}{4\lambda^{\frac{2}{q-2}}+(q-2)^2\mu(a\sigma_1)^{\frac{4-q}{q-2}}}\mu\mathcal{S}^{-2},
\end{eqnarray*}
then $\{u_n\}\subset\mathcal{N}^-$ is a bounded $(PS)$ sequence of $\mathcal{E}(u)$ at the energy value $m^-$, where $\sigma_1$ is the first eigenvalue of $-\Delta$ in $L^2(\Omega)$.
\end{lemma}
\begin{proof}
Let $w\in\mathbb{B}_1:=\{u\in\h\mid\|\nabla u\|_{L^2(\Omega)}=1\}$ and consider the function
\begin{eqnarray*}
\mathcal{F}_{n,w}(l,s)&=&a\|\nabla(lu_n+sw)\|_{L^2(\Omega)}^2-\lambda\|(lu_n+sw)\|_{L^q(\Omega)}^q\\
&&+b\|\nabla(lu_n+sw)\|_{L^2(\Omega)}^4-\mu\|(lu_n+sw)\|_{L^4(\Omega)}^4.
\end{eqnarray*}
Since $\{u_n\}\subset\mathcal{N}^-$, by a direct calculation, we can see that
\begin{eqnarray*}
\mathcal{F}_{n,w}(1,0)=0\quad\text{and}\quad\frac{\partial\mathcal{F}_{n,w}}{\partial l}(1,0)=T_{u_n}''(1)<0.
\end{eqnarray*}
It follows from the implicit function theorem that for $\forall n\in\bbn$, there exist $\delta_n>0$ and
$$
l_n(s)\in C^1([-\delta_n, \delta_n], [\frac12, \frac32])
$$
such that $\{l_n(s)u_n+sw\}\subset\mathcal{N}$.  Since $w\in\mathbb{B}_1$, by choosing $\delta_n>0$ small enough if necessary, we can also have that $\{l_n(s)u_n+sw\}\subset\mathcal{N}^-$.  Moreover, $l_n(0)=1$ and
\begin{eqnarray*}
l_n'(0)&=&-\frac{\frac{\partial\mathcal{F}_{n,w}}{\partial s}(1,0)}{\frac{\partial\mathcal{F}_{n,w}}{\partial l}(1,0)}\\
&=&-\frac{(4b\|\nabla u_n\|_{L^2(\Omega)}^2+2a)\int_{\Omega}\nabla u_n\nabla wdx-q\lambda\int_{\Omega}|u_n|^{q-2}u_nwdx-4\mu\int_{\Omega}u_n^3wdx}{T_{u_n}''(1)}.
\end{eqnarray*}
For the sake of clarity, the following proof will be divided into three claims.

{\bf Claim~1.}\quad We have $b\|\nabla u_n\|_{L^2(\Omega)}^4-\mu\|u_n\|_{L^4(\Omega)}^{4}\leq\frac{\lambda^{\frac{2}{q-2}}b}{(\mu(a\sigma_1)^{\frac{q-2}{4-q}}+\lambda^{\frac{2}{q-2}})}\|\nabla u_n\|_{L^2(\Omega)}^4$.

Indeed, by the Young and H\"older inequalities, we can see that
\begin{eqnarray*}
\lambda\|u_n\|_{L^q(\Omega)}^q&\leq&\lambda\|u_n\|_{L^2(\Omega)}^{4-q}\|u_n\|_{L^4(\Omega)}^{2q-4}\\
&\leq&a\sigma_1\|u_n\|_{L^2(\Omega)}^2+\frac{\lambda^{\frac{2}{q-2}}}{(a\sigma_1)^{\frac{4-q}{q-2}}}\|u_n\|_{L^4(\Omega)}^{4}\\
&\leq&a\|\nabla u_n\|_{L^2(\Omega)}^2+\frac{\lambda^{\frac{2}{q-2}}}{(a\sigma_1)^{\frac{4-q}{q-2}}}\|u_n\|_{L^4(\Omega)}^{4}.
\end{eqnarray*}
It follows from $\{u_n\}\subset\mathcal{N}$ that
\begin{eqnarray*}
b\|\nabla u_n\|_{L^2(\Omega)}^4\leq(\mu+\frac{\lambda^{\frac{2}{q-2}}}{(a\sigma_1)^{\frac{4-q}{q-2}}})\|u_n\|_{L^4(\Omega)}^{4},
\end{eqnarray*}
which implies
\begin{eqnarray*}
b\|\nabla u_n\|_{L^2(\Omega)}^4-\mu\|u_n\|_{L^4(\Omega)}^{4}\leq\frac{\lambda^{\frac{2}{q-2}}b}{(\mu(a\sigma_1)^{\frac{q-2}{4-q}}+\lambda^{\frac{2}{q-2}})}\|\nabla u_n\|_{L^2(\Omega)}^4.
\end{eqnarray*}

{\bf Claim~2.}\quad If
\begin{eqnarray}\label{eq0011}
b<\frac{(q-2)^2(\mu(a\sigma_1)^{\frac{4-q}{q-2}}+\lambda^{\frac{2}{q-2}})}{4\lambda^{\frac{2}{q-2}}+(q-2)^2\mu(a\sigma_1)^{\frac{4-q}{q-2}}}\mu\mathcal{S}^{-2},
\end{eqnarray}
then we have $T_{u_n}''(1)\leq-d_1<0$, where $d_1$ is a constant independent of $n$.

Indeed, if not, then there exists a subsequence of $\{u_n\}$, which is still denoted by $\{u_n\}$, such that $T_{u_n}''(1)=o_n(1)$.  It follows from $\{u_n\}\subset\mathcal{N}$ that
\begin{eqnarray}\label{eq0004}
(4-q)(b\|\nabla u_n\|_{L^2(\Omega)}^4-\mu\|u_n\|_{L^4(\Omega)}^4)=(q-2)a\|\nabla u_n\|_{L^2(\Omega)}^2+o_n(1).
\end{eqnarray}
Since $2<q<4$ and $\{u_n\}\subset\mathcal{N}^-$, by a similar argument as used in \eqref{eq0002}, we have that $\|\nabla u_n\|_{L^2(\Omega)}\geq d_0>0$, where $d_0$ is a constant.  This together with Claim~1, implies
\begin{eqnarray}\label{eq0008}
\|\nabla u_n\|_{L^2(\Omega)}^2\geq\frac{(q-2)a(\mu(a\sigma_1)^{\frac{q-2}{4-q}}+\lambda^{\frac{2}{q-2}})}{(4-q)b\lambda^{\frac{2}{q-2}}}+o_n(1).
\end{eqnarray}
On the other hand, by Lemma~\ref{lem0003} and a similar argument as used for \eqref{eq0001}, we can see from $(a)$ and \eqref{eq0008} that
\begin{eqnarray}\label{eq0009}
\frac{a^2\mathcal{S}^2}{4(\mu-b\mathcal{S}^2)}>\frac{q-2}{4q}a\|\nabla u_n\|_{L^2(\Omega)}^2+o_n(1)\geq\frac{[a(q-2)]^2(\mu(a\sigma_1)^{\frac{q-2}{4-q}}+\lambda^{\frac{2}{q-2}})}{4q(4-q)b\lambda^{\frac{2}{q-2}}}+o_n(1),
\end{eqnarray}
which together with a direct calculation, implies $b\geq\frac{(q-2)^2(\mu(a\sigma_1)^{\frac{4-q}{q-2}}+\lambda^{\frac{2}{q-2}})}{4\lambda^{\frac{2}{q-2}}+(q-2)^2\mu(a\sigma_1)^{\frac{4-q}{q-2}}}\mu\mathcal{S}^{-2}$.  It contradicts to \eqref{eq0011}.

{\bf Claim~3.}\quad We have $\{u_n\}$ is bounded in $\h$ and $\mathcal{E}'(u_n)=o_n(1)$ strongly in $H^{-1}(\Omega)$.

Indeed, for every $w\in\mathbb{B}_1$, we take $v=l_n(s)u_n+sw$ in $(b)$.  Then by the Taylor expansion and the fact that $\{u_n\}\subset\mathcal{N}^-$, we have that
\begin{eqnarray}
s\mathcal{E}'(u_n)w&=&\mathcal{E}'(u_n)((l_n(s)-1)u_n+sw)\notag\\
&\geq&-\frac1n(s+|l_n(s)-1|\|\nabla u_n\|_{L^2(\Omega)})+o(s)+o(|l_n(s)-1|\|\nabla u_n\|_{L^2(\Omega)}).\label{eq0010}
\end{eqnarray}
By \eqref{eq0009} and Lemma~\ref{lem0003}, we can see that $\{u_n\}$ is bounded in $\h$.  It follows from Claim~2 that $|l_n'(0)|\leq d_2$, where $d_2>0$ is a constant.   Since $|l_n(s)-1|\|\nabla u_n\|_{L^2(\Omega)}\to0$ as $s\to0$, by multiplying $s^{-1}$ on both side of \eqref{eq0010} and letting $s\to0^+$, we have $\mathcal{E}'(u_n)w\geq-\frac{d_3}{n}$, where $d_3>0$ is a constant.  Since $w\in\mathbb{B}_1$ is arbitrary, we must have that $\mathcal{E}'(u_n)w=o_n(1)$, where $o_n(1)$ is independent of $w\in\mathbb{B}_1$.  Thus, $\mathcal{E}'(u_n)=o_n(1)$ strongly in $H^{-1}(\Omega)$.
\end{proof}

Now, we can obtain the following.
\begin{proposition}\label{prop0002}
Let $N=4$, $2<q<4$ and $b\mathcal{S}^2<\mu$.  If we also have \eqref{eq0013},
then $(\mathcal{P}_{a,b,\lambda,\mu})$ has a solution minimizing $\mathcal{E}(u)$ on $\mathcal{N}^-$.
\end{proposition}
\begin{proof}
\quad By Lemma~\ref{lem0004} and \cite[Proposition~1.7]{N14}, one of the following two cases must happen:
\begin{enumerate}
\item[$(a)$]  $\{u_n\}$ has a subsequence which strongly converges in $H^1_0(\Omega)$;
\item[$(b)$]  there exist a function $u_0\in H^1_0(\Omega)$ which is a weak convergence of $\{u_n\}$, a number $k\in \bbn$ and further, for every $i\in\{1, 2, \cdots k\}$, a sequence of value $\{R^i_n\}_{n\in\bbn}\subset \bbr^+$, points $\{x^i_n\}_{n\in\bbn}\subset\overline{\Omega}$ and a function $v_i\in\mathcal{D}^{1,2}(\bbr^4)$ satisfying
\begin{equation*}\label{eq9907}
-\Big\{a+b\Big(\|\nabla u_0\|^2_{L^2(\Omega)}+\sum_{j=1}^k\|\nabla v_j\|^2_{L^2(\bbr^4)}\Big)\Big\}\Delta
u_0 = \lambda |u_0|^{q-2}u_0+\mu u_0^{3}, \quad \text{in}\ \Omega,
\end{equation*}
and
\begin{equation}\label{eq9908}
-\Big\{a+b\Big(\|\nabla u_0\|^2_{L^2(\Omega)}+\sum_{j=1}^k\|\nabla v_j\|^2_{L^2(\bbr^4)}\Big)\Big\}\Delta
v_i = \mu v_i^{3}, \quad \text{in}\ \bbr^4,
\end{equation}
such that up to subsequences, there hold $R^i_ndist(x^i_n,
\partial\Omega)\to\infty$,
\begin{eqnarray}\label{eq9999}
\Big\|\nabla (u_n-u_0-\sum_{i=1}^k(R^i_n)^{2}v_i(R^i_n(\cdot-x^i_n)))\Big\|_{L^2(\bbr^4)}=o_n(1),
\end{eqnarray}
\begin{equation}\label{eq91001}
\|\nabla u_n\|^2_{L^2(\Omega)}=\|\nabla u_0\|^2_{L^2(\Omega)}+\sum_{i=1}^k\|\nabla v_i\|^2_{L^2(\bbr^4)}+o_n(1)
\end{equation}
and
\begin{eqnarray}\label{eq90019}
\mathcal{E}(u_n)=\widetilde{\mathcal{E}}(u_0)+\sum_{i=1}^k\widetilde{\mathcal{E}}_{\infty}(v_i)+o_n(1),
\end{eqnarray}
where
\begin{eqnarray*}
\widetilde{\mathcal{E}}(u_0)&=&\Big\{\frac{a}{2}+\frac{b}{4}\Big(\|\nabla u_0\|^2_{L^2(\Omega)}+\sum_{j=1}^k\|\nabla v_j\|^2_{L^2(\bbr^4)}\Big)\Big\}\|\nabla u_0\|^2_{L^2(\Omega)}\notag\\
&&-\frac{\lambda}{q}\|u_0\|^q_{L^q(\Omega)}-\frac{\mu}{2^*}\|u_0\|^{2^*}_{L^{2^*}(\Omega)},\label{eq9909}
\end{eqnarray*}
and
\begin{eqnarray*}\label{eq9910}
\widetilde{\mathcal{E}}_{\infty}(v_i)=\Big\{\frac{a}{2}+\frac{b}{4}\Big(\|\nabla u_0\|^2_{L^2(\Omega)}+\sum_{j=1}^k\|\nabla v_j\|^2_{L^2(\bbr^4)}\Big)\Big\}\|\nabla v_i\|^2_{L^2(\bbr^4)}-\frac{\mu}{2^*}\|v_i\|^{2^*}_{L^{2^*}(\bbr^4)}.
\end{eqnarray*}
\end{enumerate}
Thanks to \cite[Remark~4.2]{N14}, we may assume $u_n\geq0,u_0\geq0$ and $v_i>0$.  It follows from the calculations of \cite[Remark~4.3]{N14} and \cite[Theorem~1.1]{WHL15} (see also \cite[Theorem~1.1]{LLT15}) that $\|\nabla v_i\|^2_{L^2(\bbr^4)}=(a+b\mathcal{A})\mu^{-1}\mathcal{S}^{2}$ and $\|v_i\|^{4}_{L^{4}(\bbr^4)}=[(a+b\mathcal{A})\mu^{-1}]^{2}\mathcal{S}^{2}$ for all $i$, where $\mathcal{A}=\|\nabla u_0\|^2_{L^2(\Omega)}+\sum_{i=1}^k\|\nabla v_i\|^2_{L^2(\bbr^4)}$.  Consider the functions
\begin{eqnarray*}
\mathcal{H}_{v_i}(t)=\frac{a}{2}\|\nabla v_i\|^2_{L^2(\bbr^4)}t^2+\frac{b}{4}\|\nabla v_i\|^4_{L^2(\bbr^4)}t^4-\frac{\mu}{4}\|v_i\|^{4}_{L^{4}(\bbr^4)}t^4,\quad i=1,2,\cdots,k.
\end{eqnarray*}
By \eqref{eq9908}, it is easy to see that $\mu\|v_i\|^{4}_{L^{4}(\bbr^4)}-b\|\nabla v_i\|^4_{L^2(\bbr^4)}>0$ and
\begin{eqnarray*}
t_0&=&\bigg(\frac{a\|\nabla v_i\|^2_{L^2(\bbr^4)}}{\mu\|v_i\|^{4}_{L^{4}(\bbr^4)}-b\|\nabla v_i\|^4_{L^2(\bbr^4)}}\bigg)^{\frac12}\\
&=&\bigg(\frac{a}{(\mu-b\mathcal{S}^2)(a+b\mathcal{A})}\bigg)^{\frac12}\\
&<&1.
\end{eqnarray*}
Moreover, $\mathcal{H}_{v_i}(t)$ is strictly increasing for $t\in(0, t_0)$ and strictly decreasing for $(t_0, +\infty)$.  By Lemma~\ref{lem0001} and similar arguments as used for \eqref{eq91001} and \eqref{eq90019}, we have that
\begin{eqnarray}
\mathcal{E}(u_n)&\geq&\mathcal{E}(t_0u_n)\notag\\
&=&(\frac{a}{2}+\frac{b}{4}\|\nabla u_n\|_{L^2(\Omega)}^2t_0^2)\|\nabla u_n\|_{L^2(\Omega)}^2t_0^2-\frac{\lambda}{q}\|u_n\|_{L^q(\Omega)}^qt_0^q-\frac{\mu}{4}\|u_n\|_{L^{4}(\Omega)}^{4}t_0^4\notag\\
&\geq&\mathcal{E}(t_0u_0)+\sum_{i=1}^k\mathcal{H}_{v_i}(t_0)+o_n(1)\notag\\
&=&\mathcal{E}(t_0u_0)+k\frac{a^2\mathcal{S}^2}{4(\mu-b\mathcal{S}^2)}+o_n(1).\label{eq90020}
\end{eqnarray}
It follows from Lemma~\ref{lem0003} that either $k=0$ or $\mathcal{E}(t_0u_0)<0$.
Thanks to \eqref{eq9999}, we may assume that $k\geq1$ and $\mathcal{E}(t_0u_0)<0$.  By Lemma~\ref{lem0001} once more, we must have that there exists $\tilde{t}_*(u_0)<t_0$ such that $\tilde{t}_*(u_0)u_0\in\mathcal{N}^-$.  Now, we can see from the properties of $\mathcal{H}_{v_i}(t)$, a similar argument as used for \eqref{eq90020} and Lemma~\ref{lem0001} that
\begin{eqnarray*}
m^-+o_n(1)&=&\mathcal{E}(u_n)\\
&\geq&\mathcal{E}(\tilde{t}_*(u_0)u_n)\\
&\geq&\mathcal{E}(\tilde{t}_*(u_0)u_0)+\sum_{i=1}^k\mathcal{H}_{v_i}(\tilde{t}_*(u_0))+o_n(1)\\
&>&m^-.
\end{eqnarray*}
It is impossible.  Hence, we must have the case $(a)$.  It follows from the strong maximum principle that $u_0$ is a solution of $(\mathcal{P}_{a,b,\lambda,\mu})$ minimizing $\mathcal{E}(u)$ on $\mathcal{N}^-$ under \eqref{eq0013}.
\end{proof}

We close this section by

\noindent\textbf{Proof of Theorem~\ref{thm0004}:}  It follows immediately from Propositions~\ref{prop0001} and \ref{prop0002}.
\qquad\raisebox{-0.5mm}{%
\rule{1.5mm}{4mm}}\vspace{6pt}

\section{The case $N\geq5$}

\subsection{The Nehari manifold $\mathcal{N}$}
Let us first make some observations on the Nehari manifold $\mathcal{N}$ by the fibering maps $T_u(t)=\mathcal{E}(tu)$.
\begin{lemma}\label{lem0005}
Let $N\geq5$.
\begin{enumerate}
\item[$(1)$] If $q=2$ and $\lambda<a\sigma_1$, then $\chi_u\cap\mathcal{N}\neq \emptyset$ if and only if $\mathcal{G}(u)\leq0$, where
\begin{eqnarray*}
\mathcal{G}(u)=a\|\nabla u\|_{L^2(\Omega)}^2-\lambda\|u\|_{L^2(\Omega)}^2-\frac{(4-2^*)\mu}{2}\bigg[\frac{(2^*-2)\mu\|u\|_{L^{2^*}(\Omega)}^{2^*}}{2b\|\nabla u\|_{L^2(\Omega)}^4}\bigg]^{\frac{2^*-2}{4-2^*}}\|u\|_{L^{2^*}(\Omega)}^{2^*}).
\end{eqnarray*}
Moreover, there exists a unique $0<t_*(u)<\widetilde{t}_*(u)$ such that $t_*(u)u\in\mathcal{N}^-$ for $\mathcal{G}(u)<0$ and $T_u(1)=\max_{0\leq t\leq\widetilde{t}_*(u)}T_u(t)$, where $\widetilde{t}_*(u)=\bigg[\frac{(2^*-2)\mu\|u\|_{L^{2^*}(\Omega)}^{2^*}}{2b\|\nabla u\|_{L^2(\Omega)}^4}\bigg]^{\frac{1}{4-2^*}}$.
\item[$(2)$] If $q>2$ and $\mathcal{F}(u)<0$, then there exists a unique $0<t_*(u)<\tilde{t}_0(u)$ such that $t_*(u)u\in\mathcal{N}^-$,
where
\begin{eqnarray*}
\mathcal{F}(u)&=&a\|\nabla u\|_{L^2(\Omega)}^2-\lambda\|u\|_{L^q(\Omega)}^q\bigg[\frac{(2^*-2)(2^*-q)\mu\|u\|_{L^{2^*}(\Omega)}^{2^*}}{(4-q)2b\|\nabla u\|_{L^2(\Omega)}^4}\bigg]^{\frac{q-2}{4-2^*}}\\
&&-\frac{(4-2^*)(2^*+2-q)}{2(4-q)}\bigg[\frac{(2^*-2)(2^*-q)\mu\|u\|_{L^{2^*}(\Omega)}^{2^*}}{(4-q)2b\|\nabla u\|_{L^2(\Omega)}^4}\bigg]^{\frac{2^*-2}{4-2^*}}\mu\|u\|_{L^{2^*}(\Omega)}^{2^*}
\end{eqnarray*}
and $\tilde{t}_0(u)=\bigg[\frac{(2^*-2)(2^*-q)\mu\|u\|_{L^{2^*}(\Omega)}^{2^*}}{(4-q)2b\|\nabla u\|_{L^2(\Omega)}^4}\bigg]^{\frac{1}{4-2^*}}$.  Moreover, $T_u(1)=\max_{0\leq t\leq\tilde{t}_0(u)}T_u(t)$.
\end{enumerate}
\end{lemma}
\begin{proof}
By a direct calculation, we can see that for every $u\in\h\backslash\{0\}$, we have
\begin{eqnarray*}
T_u'(t)=t(a\|\nabla u\|_{L^2(\Omega)}^2+b\|\nabla u\|_{L^2(\Omega)}^4t^2-\lambda\|u\|_{L^q(\Omega)}^qt^{q-2}-\mu\|u\|_{L^{2^*}(\Omega)}^{2^*}t^{2^*-2}).
\end{eqnarray*}

$(1)$\quad Let
\begin{eqnarray*}
T_{u,1}^*(t)=b\|\nabla u\|_{L^2(\Omega)}^4t^2-\mu\|u\|_{L^{2^*}(\Omega)}^{2^*}t^{2^*-2}.
\end{eqnarray*}
Then by a direct calculation, we can see that
\begin{eqnarray*}
(T_{u,1}^*)'(t)=2b\|\nabla u\|_{L^2(\Omega)}^4t-(2^*-2)\mu\|u\|_{L^{2^*}(\Omega)}^{2^*}t^{2^*-3}.
\end{eqnarray*}
It follows that $T_{u,1}^*(\widetilde{t}_*(u))=\min_{t\geq0}T_{u,1}^*(t)$ and $T_{u,1}^*(t)$ is strictly decreasing on $(0, \widetilde{t}_*(u))$ and strictly increasing on $(\widetilde{t}_*(u), +\infty)$.  This together with $q=2$ and $\lambda<a\sigma_1$, implies $\min_{t\geq0}\frac{T_u'(t)}{t}=\frac{T_u'(\widetilde{t}_*(u))}{\widetilde{t}_*(u)}=\mathcal{G}(u)$ and $\frac{T_u'(t)}{t}$ is strictly decreasing on $(0, \widetilde{t}_*(u))$ and strictly increasing on $(\widetilde{t}_*(u), +\infty)$.  Now, the conclusions follows immediately from the relations between $T_u(t)$ and $\mathcal{N}$ and the definition of $\mathcal{N}^-$.

$(2)$\quad Let
\begin{eqnarray*}
T_{u,1}(t)=b\|\nabla u\|_{L^2(\Omega)}^4t^2-\lambda\|u\|_{L^q(\Omega)}^qt^{q-2}-\mu\|u\|_{L^{2^*}(\Omega)}^{2^*}t^{2^*-2}.
\end{eqnarray*}
Since $2<q<2^*$, by a direct calculation, we can see that
\begin{eqnarray*}
T_{u,1}'(t)=t^{q-3}(2b\|\nabla u\|_{L^2(\Omega)}^4t^{4-q}-(2^*-2)\mu\|u\|_{L^{2^*}(\Omega)}^{2^*}t^{2^*-q}-(q-2)\lambda\|u\|_{L^q(\Omega)}^q).
\end{eqnarray*}
Let
\begin{eqnarray*}
T_{u,2}(t)=2b\|\nabla u\|_{L^2(\Omega)}^4t^{4-q}-(2^*-2)\mu\|u\|_{L^{2^*}(\Omega)}^{2^*}t^{2^*-q}.
\end{eqnarray*}
Then
\begin{eqnarray*}
T_{u,2}'(t)=t^{2^*-q-1}((4-q)2b\|\nabla u\|_{L^2(\Omega)}^4t^{4-2^*}-(2^*-q)(2^*-2)\mu\|u\|_{L^{2^*}(\Omega)}^{2^*}).
\end{eqnarray*}
Clearly, $T_{u,2}(\tilde{t}_0(u))=\min_{t\geq0}T_{u,2}(t)$ and $T_{u,2}(t)$ is strictly decreasing on $(0, \tilde{t}_0(u))$ and strictly increasing on $(\tilde{t}_0(u), +\infty)$.  Thus, there exists a unique $\tilde{t}_1(u)>\tilde{t}_0(u)$ such that $T_{u,1}'(\tilde{t}_1(u))=0$.  Moreover, $T_{u,1}(t)$ is strictly decreasing on $(0, \tilde{t}_1(u))$ and strictly increasing on $(\tilde{t}_1(u), +\infty)$.  It follows from a direct calculation that
\begin{eqnarray*}
T_{u,1}(\tilde{t}_1(u))&=&\min_{t\geq0}T_{u,1}(t)\\
&\leq& T_{u,1}(\tilde{t}_0(u))\\
&=&-\frac{(4-2^*)(2^*+2-q)}{2(4-q)}\bigg[\frac{(2^*-2)(2^*-q)\mu\|u\|_{L^{2^*}(\Omega)}^{2^*}}{(4-q)2b\|\nabla u\|_{L^2(\Omega)}^4}\bigg]^{\frac{2^*-2}{4-2^*}}\mu\|u\|_{L^{2^*}(\Omega)}^{2^*}\\
&&-\lambda\|u\|_{L^q(\Omega)}^q\bigg[\frac{(2^*-2)(2^*-q)\mu\|u\|_{L^{2^*}(\Omega)}^{2^*}}{(4-q)2b\|\nabla u\|_{L^2(\Omega)}^4}\bigg]^{\frac{q-2}{4-2^*}}.
\end{eqnarray*}
Now, the conclusions follows immediately from the relations between $T_u(t)$ and $\mathcal{N}^-$.
\end{proof}

Let $\{u_n\}\subset\h$ be a minimizing sequence of $\mathcal{S}$ and satisfy $\|u_n\|_{L^{2^*}(\Omega)}^{2^*}=\|\nabla u_n\|_{L^2(\Omega)}^2=\mathcal{S}^{\frac N2}+o_n(1)$ and $\|u_n\|_{L^{2}(\Omega)}^{2}=o_n(1)$.  Then we can easily see from a direct calculation that
\begin{eqnarray}
\mathcal{F}(u_n)&\leq& a\|\nabla u_n\|_{L^2(\Omega)}^2-\frac{(4-2^*)(2^*+2-q)}{2(4-q)}\bigg[\frac{(2^*-2)(2^*-q)\mu\|u_n\|_{L^{2^*}(\Omega)}^{2^*}}{(4-q)2b\|\nabla u_n\|_{L^2(\Omega)}^4}\bigg]^{\frac{2^*-2}{4-2^*}}\mu\|u_n\|_{L^{2^*}(\Omega)}^{2^*}\notag\\
&=&a\mathcal{S}^{\frac N2}-\frac{(4-2^*)(2^*+2-q)\mu}{2(4-q)}\bigg[\frac{(2^*-2)(2^*-q)\mu}{(4-q)2b\mathcal{S}^{\frac N2}}\bigg]^{\frac{2^*-2}{4-2^*}}\mathcal{S}^{\frac N2}+o_n(1).\label{eq0016}
\end{eqnarray}
and
\begin{eqnarray}
\mathcal{G}(u_n)&=&a\|\nabla u_n\|_{L^2(\Omega)}^2-\lambda\|u_n\|_{L^2(\Omega)}^2-\frac{(4-2^*)\mu}{2}\bigg[\frac{(2^*-2)\mu\|u_n\|_{L^{2^*}(\Omega)}^{2^*}}{2b\|\nabla u_n\|_{L^2(\Omega)}^4}\bigg]^{\frac{2^*-2}{4-2^*}}\|u_n\|_{L^{2^*}(\Omega)}^{2^*}\notag\\
&=&a\mathcal{S}^{\frac N2}-\frac{(4-2^*)\mu}{2}\bigg[\frac{(2^*-2)\mu}{2b\mathcal{S}^{\frac N2}}\bigg]^{\frac{2^*-2}{4-2^*}}\mathcal{S}^{\frac N2}+o_n(1)\label{eq60016}
\end{eqnarray}
Thus, if the conditions $(D1)$ and $(D2)$ hold, then $\mathcal{F}(u_n)<0$ and $\mathcal{G}(u_n)<0$ for $n$ large enough.  Hence, by Lemma~\ref{lem0005}, we have the following.
\begin{lemma}\label{lem0006}
Let $N\geq5$.  Then $\mathcal{N}^-$ is a nonempty set under one of the following two cases
\begin{enumerate}
\item[$(a)$] $q=2$, $0<\lambda<a\sigma_1$ and the condition $(D1)$ holds;
\item[$(b)$] $q>2$ and the condition $(D2)$ holds.
\end{enumerate}
\end{lemma}

By Lemma~\ref{lem0006}, we can see that $m^-=\inf_{\mathcal{N}^-}\mathcal{E}(u)$ is well defined.  In what follows, let us give some estimates on $m^-$.  The following lemma is useful in estimating $m^-$ and proving the local compactness of the $(PS)$ sequence.
\begin{lemma}\label{lem6001}
Let $N\geq5$ and the condition $(D1)$ holds.  Then there exist unique $\tilde{t}_2<\tilde{t}_3<\tilde{t}_4$ such that $\tilde{t}_2$ is the unique local maximum point of $g(t)$ in $\bbr^+$, $\tilde{t}_4$ is the unique local minimum point of $g(t)$ in $\bbr^+$ and $g(\tilde{t}_2)=\max_{0\leq t<\tilde{t}_4}g(t)$, where $\tilde{t}_3=\bigg[\frac{(2^*-2)\mu}{2b\mathcal{S}^{\frac N2}}\bigg]^{\frac{1}{4-2^*}}$ and $g(t)=\frac{a}{2}\mathcal{S}^{\frac N2}t^2+\frac{b}{4}\mathcal{S}^{N}t^4-\frac{\mu}{2^*}\mathcal{S}^{\frac N2}t^{2^*}$.
\end{lemma}
\begin{proof}
By a similar argument as used in $(1)$ of Lemma~\ref{lem0005}, we can see that $\min_{t\geq0}\frac{g'(t)}{t}=\frac{g'(\tilde{t}_3)}{\tilde{t}_3}$ and $\frac{g'(t)}{t}$ is strictly decreasing on $(0, \tilde{t}_3)$ and strictly increasing on $(\tilde{t}_3, +\infty)$.  Note that $\frac{g'(\tilde{t}_3)}{\tilde{t}_3}<0$ under the condition $(D1)$.  Thus, there exist unique $\tilde{t}_2<\tilde{t}_3<\tilde{t}_4$ such that $\tilde{t}_2$ is the unique local maximum point of $g(t)$ in $\bbr^+$, $\tilde{t}_4$ is the unique local minimum point of $g(t)$ in $\bbr^+$ and $g(\tilde{t}_2)=\max_{0\leq t<\tilde{t}_4}g(t)$.
\end{proof}

Now, we can give some estimates on $m^-$.
\begin{lemma}\label{lem0007}
Let $N\geq5$.  Then we have $0<m^-<g(\tilde{t}_2)$ under one of the following two cases
\begin{enumerate}
\item[$(a)$] $q=2$, $0<\lambda<a\sigma_1$ and the condition $(D1)$ hold;
\item[$(b)$] $q>2$ and the conditions $(D_1)$--$(D2)$ hold.
\end{enumerate}
\end{lemma}
\begin{proof}
Let us first estimate the low bound of $m^-$ for $q=2$.  Suppose $u\in\mathcal{N}^-$, then by the definition of $\mathcal{N}^-$, we can see that
\begin{eqnarray*}
a\|\nabla u\|_{L^2(\Omega)}^2-\lambda\|u\|_{L^2(\Omega)}^2+b\|\nabla u\|_{L^2(\Omega)}^4-\mu\|u\|_{L^{2^*}(\Omega)}^{2^*}=0
\end{eqnarray*}
and
\begin{eqnarray*}
a\|\nabla u\|_{L^2(\Omega)}^2-\lambda\|u\|_{L^2(\Omega)}^2+3b\|\nabla u\|_{L^2(\Omega)}^4-(2^*-1)\mu\|u\|_{L^{2^*}(\Omega)}^{2^*}<0.
\end{eqnarray*}
It follows that
\begin{eqnarray*}
(4-2^*)b\|\nabla u\|_{L^2(\Omega)}^4<(2^*-2)(a\|\nabla u\|_{L^2(\Omega)}^2-\lambda\|u\|_{L^2(\Omega)}^2),
\end{eqnarray*}
which together with $u\in\mathcal{N}^-\subset\mathcal{N}$ and $0<\lambda<a\sigma_1$, implies
\begin{eqnarray}
\mathcal{E}(u)&=&\mathcal{E}(u)-\frac{1}{2^*}\mathcal{E}'(u)u\notag\\
&=&\frac{2^*-2}{22^*}(a\|\nabla u\|_{L^2(\Omega)}^2-\lambda\|u\|_{L^2(\Omega)}^2)-\frac{4-2^*}{42^*}b\|\nabla u\|_{L^2(\Omega)}^4\notag\\
&>&\frac{2^*-2}{42^*}(a\|\nabla u\|_{L^2(\Omega)}^2-\lambda\|u\|_{L^2(\Omega)}^2)\notag\\
&\geq&\frac{(2^*-2)a}{42^*}(1-\frac{\lambda}{a\sigma_1})\|\nabla u\|_{L^2(\Omega)}^2.\label{eq60015}
\end{eqnarray}
On the other hand, by modifying the argument as used for \eqref{eq5002} trivially, we can show that $\|\nabla u\|_{L^2(\Omega)}\geq d_0$, where $d_0>0$ is a constant.  Thus, by \eqref{eq60015}, $m^-\geq\frac{(2^*-2)a}{42^*}(1-\frac{\lambda}{a\sigma_1})d_0^2>0$.  Next, we estimate the low bound of $m^-$ for $q>2$.  Suppose $u\in\mathcal{N}^-$, then by the definition of $\mathcal{N}^-$, we can see that
\begin{eqnarray*}
a\|\nabla u\|_{L^2(\Omega)}^2+b\|\nabla u\|_{L^2(\Omega)}^4-\lambda\|u\|_{L^q(\Omega)}^q-\mu\|u\|_{L^{2^*}(\Omega)}^{2^*}=0
\end{eqnarray*}
and
\begin{eqnarray*}
a\|\nabla u\|_{L^2(\Omega)}^2+3b\|\nabla u\|_{L^2(\Omega)}^4-(q-1)\lambda\|u\|_{L^q(\Omega)}^q-(2^*-1)\mu\|u\|_{L^{2^*}(\Omega)}^{2^*}<0.
\end{eqnarray*}
It follows that
\begin{eqnarray*}
(4-q)\lambda\|u\|_{L^q(\Omega)}^q+(4-2^*)\mu\|u\|_{L^{2^*}(\Omega)}^{2^*}<2a\|\nabla u\|_{L^2(\Omega)}^2,
\end{eqnarray*}
which together with $u\in\mathcal{N}^-\subset\mathcal{N}$, implies
\begin{eqnarray}
\mathcal{E}(u)&=&\mathcal{E}(u)-\frac{1}{4}\mathcal{E}'(u)u\notag\\
&=&\frac{1}{4}a\|\nabla u\|_{L^2(\Omega)}^2-\frac{4-q}{4q}\lambda\|u\|_{L^q(\Omega)}^q-\frac{4-2^*}{42^*}\mu\|u\|_{L^{2^*}(\Omega)}^{2^*}\notag\\
&\geq&\frac{1}{4}a\|\nabla u\|_{L^2(\Omega)}^2-\frac{1}{4q}\bigg((4-q)\lambda\|u\|_{L^q(\Omega)}^q+(4-2^*)\mu\|u\|_{L^{2^*}(\Omega)}^{2^*}\bigg)\notag\\
&>&\frac{q-2}{4q}a\|\nabla u\|_{L^2(\Omega)}^2.\label{eq0015}
\end{eqnarray}
On the other hand, since $u\in\mathcal{N}^-\subset\mathcal{N}$ and $b>0$, we can see from the definitions of $\mathcal{S}_q$ and $\mathcal{S}$ that
\begin{eqnarray}
a\|\nabla u\|_{L^2(\Omega)}^2&\leq&\lambda\|u\|_{L^q(\Omega)}^q+\mu\|u\|_{L^{2^*}(\Omega)}^{2^*}\notag\\
&\leq&\lambda\mathcal{S}_q^{-\frac{q}{2}}\|\nabla u\|_{L^2(\Omega)}^q+\mu\mathcal{S}^{-\frac{2^*}{2}}\|\nabla u\|_{L^2(\Omega)}^{2^*}.\label{eq0014}
\end{eqnarray}
Note that $2<q<2^*$, we must have from \eqref{eq0014} that $\|\nabla u\|_{L^2(\Omega)}\geq d_0>0$, where $d_0$ is a constant.  Now, by \eqref{eq0015}, we can see that $m^-\geq\frac{a(q-2)d_0^2}{4q}>0$.  We finally estimate the up bound of $m^-$ for all $2\leq q$.  Let
\begin{eqnarray*}
w_\ve=\frac{[N(N-2)]^{\frac{N-2}{4}}\ve^{\frac{N-2}{2}}}{(\ve^2+|x-x_0|^2)^{\frac{N-2}{2}}}\varphi.
\end{eqnarray*}
where $\ve>0$, $x_0\in\Omega$ and $\varphi\in C^\infty_0(\Omega)$ satisfying $\varphi\equiv1$ on a ball centered at $x_0$ and contained in $\Omega$.  Then it is well known that
$\|\nabla w_\ve\|_{L^2(\Omega)}^2=\mathcal{S}^{\frac N2}+O(\ve^{N-2})$, $\|w_\ve\|_{L^{2^*}(\Omega)}^{2^*}=\mathcal{S}^{\frac N2}+O(\ve^{N})$ and $\|w_\ve\|_{L^q(\Omega)}^q\geq d_1\ve^{N+q-\frac{qN}{2}}$ for $\ve$ small enough, where $d_1>0$ is a constant.  By similar arguments as used in \eqref{eq0016} and \eqref{eq60016}, we can see that $\mathcal{F}(w_\ve)<0$ for $\ve$ small enough under the condition $(D2)$ and $\mathcal{G}(w_\ve)<0$ for $\ve$ small enough under the condition $(D1)$.  It follows from Lemma~\ref{lem0005} that there exists a unique $0<t_*(w_\ve)$ such that $t_*(w_\ve)w_\ve\in\mathcal{N}^-$ for $\ve$ small enough, where $t_*(w_\ve)<\widetilde{t}_*(w_\ve)$ for $q=2$ and $t_*(w_\ve)<\tilde{t}_0(w_\ve)$ for $q>2$.  Thus, by the definitions of $\widetilde{t}_*(w_\ve)$ and $\tilde{t}_0(w_\ve)$, we have
\begin{eqnarray}
t_*(w_\ve)&<&\bigg[\frac{(2^*-2)\mu\mathcal{S}^{\frac N2}+O(\ve^N)}{2b\mathcal{S}^{N}+O(\ve^{2N-4})}\bigg]^{\frac{1}{4-2^*}}\notag\\
&=&\bigg[\frac{(2^*-2)\mu}{2b\mathcal{S}^{\frac N2}}\bigg]^{\frac{1}{4-2^*}}+O(\ve^{\frac{N}{4-2^*}})\label{eq60017}
\end{eqnarray}
for $q=2$ and
\begin{eqnarray}
t_*(w_\ve)&<&\bigg[\frac{(2^*-2)(2^*-q)\mu\mathcal{S}^{\frac N2}+O(\ve^N)}{(4-q)2b\mathcal{S}^{N}+O(\ve^{2N-4})}\bigg]^{\frac{1}{4-2^*}}\notag\\
&=&\bigg[\frac{(2^*-2)(2^*-q)\mu}{(4-q)2b\mathcal{S}^{\frac N2}}\bigg]^{\frac{1}{4-2^*}}+O(\ve^{\frac{N}{4-2^*}})\label{eq0017}
\end{eqnarray}
for $q>2$.  Since $2^*<4$ for $N\geq5$, by \eqref{eq60017}--\eqref{eq0017} and Lemma~\ref{lem6001}, we can see that $t_*(w_\ve)<\tilde{t}_4$ for $\ve$ small enough.  It follows from $2\leq q$, $N\geq5$ and Lemma~\ref{lem6001} once more that
\begin{eqnarray*}
m^-&\leq&\mathcal{E}(\tilde{t}_*(w_\ve)w_\ve)\\
&\leq& g(\tilde{t}_*(w_\ve))-d_2\ve^{N+q-\frac{qN}{2}}+O(\ve^{N-2})\\
&\leq& g(\tilde{t}_2)-d_2\ve^{N+q-\frac{qN}{2}}+O(\ve^{N-2})\\
&<&g(\tilde{t}_2)
\end{eqnarray*}
for $\ve$ small enough, where $d_2>0$ is a constant.
\end{proof}

\subsection{The subcase $q=2$}
Since $\mathcal{N}^-\not=\emptyset$, by the Ekeland's principle, there exists $\{u_n\}\subset\mathcal{N}^-$ such that
\begin{enumerate}
\item[$(a)$] $\mathcal{E}(u_n)=m^-+o_n(1)$;
\item[$(b)$] $\mathcal{E}(v)-\mathcal{E}(u_n)\geq-\frac1n\|\nabla (v-u_n)\|_{L^2(\Omega)}$ for all $v\in\mathcal{N}^-$.
\end{enumerate}
In what follows, we will show that under some further conditions on the parameters $a,b,\lambda,\mu$, we actually have that $\{u_n\}\subset\mathcal{N}^-$ is a bounded $(PS)$ sequence of $\mathcal{E}(u)$ at the energy value $m^-$ for $q=2$.
\begin{lemma}\label{lem60008}
Let $N\geq5$ and $2=q$.  If the conditions $(D0)$--$(D1)$ hold, then $\{u_n\}\subset\mathcal{N}^-$ is a bounded $(PS)$ sequence of $\mathcal{E}(u)$ at the energy value $m^-$.
\end{lemma}
\begin{proof}
By making some trivial modifications in the proof of Lemma~\ref{lem0004}, we can show that there exist $\{\delta_n\}\subset \bbr^+$ and
$$
\{l_n(s)\}\subset C^1([-\delta_n, \delta_n], [\frac12, \frac32])
$$
such that $\{l_n(s)u_n+sw\}\subset\mathcal{N}^-$.  Moreover, $l_n(0)=1$ and
\begin{eqnarray*}
l_n'(0)=\frac{(4b\|\nabla u_n\|_{L^2(\Omega)}^2+2a)\int_{\Omega}\nabla u_n\nabla wdx-2\lambda\int_{\Omega}u_nwdx-2^*\mu\int_{\Omega}|u_n|^{2^*-2}u_nwdx}{-T_{u_n}''(1)}.
\end{eqnarray*}
We claim that $T_{u_n}''(1)\leq-d_0<0$, where $d_0>0$ is a constant independent of $n$.  Indeed, if not, then there exists a subsequence of $\{u_n\}$, which is still denoted by $\{u_n\}$, such that $T_{u_n}''(1)=o_n(1)$.  It follows from $\{u_n\}\subset\mathcal{N}^-\subset\mathcal{N}$ that
\begin{eqnarray*}
(4-2^*)b\|\nabla u_n\|_{L^2(\Omega)}^4=(2^*-2)(a\|\nabla u_n\|_{L^2(\Omega)}^2-\lambda\|u_n\|_{L^2(\Omega)}^2)+o_n(1),
\end{eqnarray*}
which together with a similar argument as used for \eqref{eq5002}, implies
\begin{eqnarray}\label{eq0018}
\|\nabla u_n\|_{L^2(\Omega)}^2\geq\frac{(2^*-2)a}{(4-2^*)b}(1-\frac{\lambda}{a\sigma_1})+o_n(1).
\end{eqnarray}
On the other hand, since the condition $(D1)$ holds, we have from Lemma~\ref{lem0007} that $m^-<g(\tilde{t}_2)$.  Note that $\tilde{t}_2$ is the unique local maximum point of $g(t)$ in $\bbr^+$.  Thus, $g''(\tilde{t}_2)<0$, which implies $2b\mathcal{S}^{N}\tilde{t}_2^4<(2^*-2)\mu \mathcal{S}^{\frac N2}\tilde{t}_2^{2^*}$.  It follows from $\frac{(2^*+2)(4-2^*)}{82^*}>0$ that
\begin{eqnarray*}
m^-&<&g(\tilde{t}_2)\\
&<&\frac{a}{2}\mathcal{S}^{\frac N2}\tilde{t}_2^2-\frac{(2^*+2)(4-2^*)\mu}{82^*}\mathcal{S}^{\frac N2}\tilde{t}_2^{2^*}\\
&\leq&\sup_{t\geq0}(\frac{a}{2}\mathcal{S}^{\frac N2}t^2-\frac{(2^*+2)(4-2^*)\mu}{82^*}\mathcal{S}^{\frac N2}t^{2^*})\\
&=&\frac{a}{N}\bigg[\frac{8a}{(2^*+2)(4-2^*)\mu}\bigg]^{\frac{2}{2^*-2}}\mathcal{S}^{\frac N2}.
\end{eqnarray*}
By \eqref{eq60015} and \eqref{eq0018}, we have
\begin{eqnarray*}
\bigg[\frac{8a}{(2^*+2)(4-2^*)\mu}\bigg]^{\frac{2}{2^*-2}}\mathcal{S}^{\frac N2}\geq\frac{N(2^*-2)^2a}{42^*(4-2^*)b}(1-\frac{\lambda}{a\sigma_1}),
\end{eqnarray*}
which contradicts to the condition $(D0)$.  Now, by a similar argument as used in Claim~3 of Lemma~\ref{lem0004}, we can show that $\{u_n\}$ is a bounded $(PS)$ sequence of $\mathcal{E}(u)$ at the energy value $m^-$, which completes the proof.
\end{proof}

Now, we can obtain the following.
\begin{proposition}\label{prop0003}
Let $N\geq5$ and $q=2$.  If the conditions $(D0)$--$(D1)$ hold, then $(\mathcal{P}_{a,b,\lambda,\mu})$ has a solution minimizing $\mathcal{E}(u)$ on $\mathcal{N}^-$.  Moreover, Problem~$(\mathcal{P}_{a,b,\lambda,\mu})$ has no solutions under the condition
\begin{eqnarray*}
a(1-\frac{\lambda}{a\sigma_1})\bigg[\frac{a\mathcal{S}^{\frac{2^*}{2}}}{\mu}(1-\frac{\lambda}{a\sigma_1})\bigg]^{\frac{2}{2^*-2}}>\bigg[\mu\mathcal{S}^{-\frac{2^*}{2}}b^{-\frac{2^*}{4}}\bigg]^{\frac{4}{4-2^*}}.
\end{eqnarray*}
\end{proposition}
\begin{proof}
By Lemma~\ref{lem60008} and a similar argument of \cite[Proposition~1.7]{N14}, one of the following two cases must happen:
\begin{enumerate}
\item[$(a)$]  $\{u_n\}$ has a subsequence which strongly converges in $H^1_0(\Omega)$;
\item[$(b)$]  there exist a function $u_0\in H^1_0(\Omega)$ which is a weak convergence of $\{u_n\}$, a number $k\in \bbn$ and further, for every $i\in\{1, 2, \cdots k\}$, a sequence of value $\{R^i_n\}_{n\in\bbn}\subset \bbr^+$, points $\{x^i_n\}_{n\in\bbn}\subset\overline{\Omega}$ and a function $v_i\in\mathcal{D}^{1,2}(\bbr^N)$ satisfying
\begin{equation*}\label{eq07}
-\Big\{a+b\Big(\|\nabla u_0\|^2_{L^2(\Omega)}+\sum_{j=1}^k\|\nabla v_j\|^2_{L^2(\bbr^N)}\Big)\Big\}\Delta
u_0 = \lambda u_0+\mu |u_0|^{2^*-2}u_0, \quad \text{in}\ \Omega,
\end{equation*}
and
\begin{equation}\label{eq08}
-\Big\{a+b\Big(\|\nabla u_0\|^2_{L^2(\Omega)}+\sum_{j=1}^k\|\nabla v_j\|^2_{L^2(\bbr^N)}\Big)\Big\}\Delta
v_i = \mu |v_i|^{2^*-2}v_i, \quad \text{in}\ \bbr^N,
\end{equation}
such that up to subsequences, there hold $R^i_ndist(x^i_n,
\partial\Omega)\to\infty$,
\begin{eqnarray}\label{eq9998}
\Big\|\nabla (u_n-u_0-\sum_{i=1}^k(R^i_n)^{\frac{N-2}{2}}v_i(R^i_n(\cdot-x^i_n)))\Big\|_{L^2(\bbr^N)}=o_n(1),
\end{eqnarray}
\begin{equation}\label{eq1001}
\|\nabla u_n\|^2_{L^2(\Omega)}=\|\nabla u_0\|^2_{L^2(\Omega)}+\sum_{i=1}^k\|\nabla v_i\|^2_{L^2(\bbr^N)}+o_n(1)
\end{equation}
and
\begin{eqnarray}\label{eq0019}
\mathcal{E}(u_n)=\widetilde{\mathcal{E}}(u_0)+\sum_{i=1}^k\widetilde{\mathcal{E}}_{\infty}(v_i)+o_n(1),
\end{eqnarray}
where
\begin{eqnarray*}
\widetilde{\mathcal{E}}(u_0)&=&\Big\{\frac{a}{2}+\frac{b}{4}\Big(\|\nabla u_0\|^2_{L^2(\Omega)}+\sum_{j=1}^k\|\nabla v_j\|^2_{L^2(\bbr^N)}\Big)\Big\}\|\nabla u_0\|^2_{L^2(\Omega)}\notag\\
&&-\frac{\lambda}{2}\|u_0\|^2_{L^2(\Omega)}-\frac{\mu}{2^*}\|u_0\|^{2^*}_{L^{2^*}(\Omega)},\label{eq09}
\end{eqnarray*}
and
\begin{eqnarray*}\label{eq10}
\widetilde{\mathcal{E}}_{\infty}(v_i)=\Big\{\frac{a}{2}+\frac{b}{4}\Big(\|\nabla u_0\|^2_{L^2(\Omega)}+\sum_{j=1}^k\|\nabla v_j\|^2_{L^2(\bbr^N)}\Big)\Big\}\|\nabla v_i\|^2_{L^2(\bbr^N)}-\frac{\mu}{2^*}\|v_i\|^{2^*}_{L^{2^*}(\bbr^N)}.
\end{eqnarray*}
\end{enumerate}
Thanks to \cite[Remark~4.2]{N14}, we may assume $u_n\geq0,u_0\geq0$ and $v_i>0$.  It follows from the calculations of \cite[Remark~4.3]{N14} and \cite[Theorem~1.1]{WHL15} (see also \cite[Theorem~1.1]{LLT15}) that $\|\nabla v_i\|^2_{L^2(\bbr^N)}=[(a+b\mathcal{A})\mu^{-1}]^{\frac{N-2}{2}}\mathcal{S}^{\frac N2}$ and $\|v_i\|^{2^*}_{L^{2^*}(\bbr^N)}=[(a+b\mathcal{A})\mu^{-1}]^{\frac{N}{2}}\mathcal{S}^{\frac N2}$ for all $i$, where $\mathcal{A}=\|\nabla u_0\|^2_{L^2(\Omega)}+\sum_{i=1}^k\|\nabla v_i\|^2_{L^2(\bbr^N)}$.  Thus, we can see from \eqref{eq08} that
\begin{eqnarray*}
&&g'([(a+b\mathcal{A})\mu^{-1}]^{\frac{N-2}{4}})[(a+b\mathcal{A})\mu^{-1}]^{\frac{N-2}{4}}\\
&=&a[(a+b\mathcal{A})\mu^{-1}]^{\frac{N-2}{2}}\mathcal{S}^{\frac N2}+b[(a+b\mathcal{A})\mu^{-1}]^{N-2}\mathcal{S}^{N}-\mu[(a+b\mathcal{A})\mu^{-1}]^{\frac{N}{2}}\mathcal{S}^{\frac N2}\\
&=&a\|\nabla v_i\|^2_{L^2(\bbr^N)}+b\|\nabla v_i\|^4_{L^2(\bbr^N)}-\mu\|v_i\|^{2^*}_{L^{2^*}(\bbr^N)}\\
&<&0,
\end{eqnarray*}
where $g(t)$ is given by Lemma~\ref{lem6001}.
Since the condition $(D1)$ holds, we can see from Lemma~\ref{lem6001} once more that there exists $\hat{t}\leq1$ such that $[(a+b\mathcal{A})\mu^{-1}]^{\frac{N-2}{4}}\hat{t}=\tilde{t}_2$, where $\tilde{t}_2$ is the unique local maximum point of $g(t)$ in $\bbr^+$ and given by Lemma~\ref{lem6001}.  Now by \eqref{eq1001}, Lemma~\ref{lem0005} and a similar argument as used for \eqref{eq0019} (cf. \cite{N14}), we can see that
\begin{eqnarray}
\mathcal{E}(u_n)&\geq&\mathcal{E}(\hat{t}u_n)\notag\\
&=&(\frac{a}{2}+\frac{b}{4}\|\nabla u_n\|_{L^2(\Omega)}^2\hat{t}^2)\|\nabla u_n\|_{L^2(\Omega)}^2\hat{t}^2-\frac{\lambda}{q}\|u_n\|_{L^2(\Omega)}^2\hat{t}^2-\frac{\mu}{2^*}\|u_n\|_{L^{2^*}(\Omega)}^{2^*}\hat{t}^{2^*}\notag\\
&\geq&\mathcal{E}(\hat{t}u_0)+\sum_{i=1}^k(\frac{a}{2}\|\nabla v_i\|^2_{L^2(\bbr^N)}\hat{t}^2+\frac{b}{4}\|\nabla v_i\|^4_{L^2(\bbr^N)}\hat{t}^4-\frac{\mu}{2^*}\|v_i\|^{2^*}_{L^{2^*}(\bbr^N)}\hat{t}^{2^*})+o_n(1)\notag\\
&=&\mathcal{E}(\hat{t}u_0)+kg(\tilde{t}_2)+o_n(1).\label{eq0020}
\end{eqnarray}
Thus, thanks to Lemma~\ref{lem0007}, we must have that either $k=0$ or $\mathcal{E}(\hat{t}u_0)<0$.  Thanks to \eqref{eq9998}, we may assume that $k\geq1$ and $\mathcal{E}(\hat{t}u_0)<0$.  Since the condition $(D1)$ holds, it follows from Lemma~\ref{lem0005} once more that there exists $\tilde{t}_*(u_0)\leq \hat{t}$ such that $\tilde{t}_*(u_0)u_0\in\mathcal{N}^-$.  Now, we can see from the choices of $\tilde{t}_*(u_0)$ and $\hat{t}$, a similar argument as used for \eqref{eq0020} and Lemma~\ref{lem0005} that
\begin{eqnarray*}
m^-+o_n(1)&=&\mathcal{E}(u_n)\\
&\geq&\mathcal{E}(\tilde{t}_*(u_0)u_n)\\
&\geq&\mathcal{E}(\tilde{t}_*(u_0)u_0)+kg([(a+b\mathcal{A})\mu^{-1}]^{\frac{N-2}{4}}\tilde{t}_*(u_0))+o_n(1)\\
&>&m^-.
\end{eqnarray*}
It is impossible.  Hence, we must have the case $(a)$, which together with the strong maximum principle, implies Problem~$(\mathcal{P}_{a,b,\lambda,\mu})$ has a solution minimizes $\mathcal{E}(u)$ on $\mathcal{N}^-$.  We finish the proof by showing $(\mathcal{P}_{a,b,\lambda,\mu})$ has no solutions under the condition
\begin{eqnarray*}
a(1-\frac{\lambda}{a\sigma_1})\bigg[\frac{a\mathcal{S}^\frac{2^*}{2}}{\mu}(1-\frac{\lambda}{a\sigma_1})\bigg]^{\frac{2}{2^*-2}}>\bigg[\mu\mathcal{S}^{-\frac{2^*}{2}}b^{-\frac{2^*}{4}}\bigg]^{\frac{4}{4-2^*}}.
\end{eqnarray*}
Indeed, suppose $u$ is a solution of $(\mathcal{P}_{a,b,\lambda,\mu})$, then by a similar argument as used for \eqref{eq5002}, we can see that
\begin{eqnarray*}
\|\nabla u\|_{L^2(\Omega)}^2\geq\bigg[\frac{a\mathcal{S}^{\frac{2^*}{2}}}{\mu}(1-\frac{\lambda}{a\sigma_1})\bigg]^{\frac{2}{2^*-2}}.
\end{eqnarray*}
It follows from $u\in\mathcal{N}$, the definition of $\sigma_1$ and the Young inequality that
\begin{eqnarray*}
0&\geq& a(1-\frac{\lambda}{a\sigma_1})\|\nabla u\|_{L^2(\Omega)}^2+b\|\nabla u\|_{L^2(\Omega)}^4-\mu\|u\|_{L^{2^*}(\Omega)}^{2^*}\\
&\geq&a(1-\frac{\lambda}{a\sigma_1})\bigg[\frac{a\mathcal{S}^{\frac{2^*}{2}}}{\mu}(1-\frac{\lambda}{a\sigma_1})\bigg]^{\frac{2}{2^*-2}}-\bigg[\mu\mathcal{S}^{-\frac{2^*}{2}}b^{-\frac{2^*}{4}}\bigg]^{\frac{4}{4-2^*}}\\
&>0&,
\end{eqnarray*}
which is a contradiction.
\end{proof}

\subsection{The subcase $q>2$}
Since $\mathcal{N}^-\not=\emptyset$, by the Ekeland's principle, there exists $\{u_n\}\subset\mathcal{N}^-$ such that
\begin{enumerate}
\item[$(a)$] $\mathcal{E}(u_n)=m^-+o_n(1)$;
\item[$(b)$] $\mathcal{E}(v)-\mathcal{E}(u_n)\geq-\frac1n\|\nabla (v-u_n)\|_{L^2(\Omega)}$ for all $v\in\mathcal{N}^-$.
\end{enumerate}
In what follows, we will show that under some further conditions on the parameters $a,b,\lambda,\mu$, we actually have that $\{u_n\}\subset\mathcal{N}^-$ is a bounded $(PS)$ sequence of $\mathcal{E}(u)$ at the energy value $m^-$ for $q>2$.
\begin{lemma}\label{lem0008}
Let $N\geq5$ and $2<q<2^*$.  If the conditions $(D1)$--$(D3)$ hold, then $\{u_n\}\subset\mathcal{N}^-$ is a bounded $(PS)$ sequence of $\mathcal{E}(u)$ at the energy value $m^-$.
\end{lemma}
\begin{proof}
The proof is similar to that of Lemma~\ref{lem60008} and we only point out the difference.  Indeed, by a similar argument as used for \eqref{eq0004}, we can show that $\|\nabla u_n\|_{L^2(\Omega)}^2\geq\frac{(q-2)a}{(4-q)b}+o_n(1)$ if $T_{u_n}''(1)=o_n(1)$.  Now, since the conditions $(D1)$--$(D2)$ hold, by Lemma~\ref{lem0007}, we can show that $\frac{4q}{N(q-2)}\mathcal{S}^{\frac N2}\bigg[\frac{8a}{(2^*+2)(4-2^*)\mu}\bigg]^{\frac{2}{2^*-2}}\geq\frac{(q-2)a}{(4-q)b}$ if $T_{u_n}''(1)=o_n(1)$, which contradicts to the condition $(D3)$.  Thus, we must have that $T_{u_n}''(1)<-d_0$, where $d_0>0$ is a constant.  Now, by a similar argument as used in the proof of Lemma~\ref{lem60008}, we can show that $\{u_n\}\subset\mathcal{N}^-$ is a bounded $(PS)$ sequence of $\mathcal{E}(u)$ at the energy value $m^-$.
\end{proof}

Now, we can obtain the following.
\begin{proposition}\label{prop0004}
Let $N\geq5$ and $2<q<2^*$.  If the conditions $(D1)$--$(D3)$ hold, then $(\mathcal{P}_{a,b,\lambda,\mu})$ has a solution minimizing $\mathcal{E}(u)$ on $\mathcal{N}^-$.  Moreover, Problem~$(\mathcal{P}_{a,b,\lambda,\mu})$ has no solutions under the condition
\begin{eqnarray*}
b\geq\bigg(\lambda \mathcal{S}_q^{-\frac{q}{2}}(\frac{2}{a})^{4-q}\bigg)^{\frac{2}{q-2}}+\bigg(\mu \mathcal{S}^{-\frac{2^*}{2}}(\frac{2}{a})^{4-2^*}\bigg)^{\frac{2}{2^*-2}}.
\end{eqnarray*}
\end{proposition}
\begin{proof}
Since Lemma~\ref{lem0008} holds under the conditions $(D1)$--$(D3)$, by modifying the proof of Proposition~\ref{prop0003} trivially, we can show that $(\mathcal{P}_{a,b,\lambda,\mu})$ has a solution minimizing $\mathcal{E}(u)$ on $\mathcal{N}^-$ under the conditions $(D1)$--$(D3)$.  We finish the proof by obtaining a nonexistence result of $(\mathcal{P}_{a,b,\lambda,\mu})$ in the cases $N\geq5$ for $q>2$.  Indeed, let $u$ be a solution of $(\mathcal{P}_{a,b,\lambda,\mu})$, then by the definitions of $\mathcal{S}_q$ and $\mathcal{S}$ and the Young inequality, we can see that
\begin{eqnarray*}
0&=&a\|\nabla u\|^2_{L^2(\Omega)}+b\|\nabla u\|^4_{L^2(\Omega)}-\lambda\|u\|^q_{L^q(\Omega)}-\mu\|u\|^{2^*}_{L^{2^*}(\Omega)}\\
&\geq&a\|\nabla u\|^2_{L^2(\Omega)}+b\|\nabla u\|^4_{L^2(\Omega)}-\lambda\mathcal{S}_q^{-\frac{q}{2}}\|\nabla u\|^q_{L^q(\Omega)}-\mu\mathcal{S}^{-\frac{2^*}{2}}\|u\|^{2^*}_{L^{2^*}(\Omega)}\\
&\geq&\bigg(b-\bigg(\lambda \mathcal{S}_q^{-\frac{q}{2}}(\frac{2}{a})^{4-q}\bigg)^{\frac{2}{q-2}}-\bigg(\mu \mathcal{S}^{-\frac{2^*}{2}}(\frac{2}{a})^{4-2^*}\bigg)^{\frac{2}{2^*-2}}\bigg)\|\nabla u\|^4_{L^2(\Omega)}.
\end{eqnarray*}
Therefore, if $b\geq\bigg(\lambda \mathcal{S}_q^{-\frac{q}{2}}(\frac{2}{a})^{4-q}\bigg)^{\frac{2}{q-2}}+\bigg(\mu \mathcal{S}^{-\frac{2^*}{2}}(\frac{2}{a})^{4-2^*}\bigg)^{\frac{2}{2^*-2}}$, then we must have $u=0$, which implies $(\mathcal{P}_{a,b,\lambda,\mu})$ has no solutions.
\end{proof}

We close this section by

\noindent\textbf{Proof of Theorem~\ref{thm0005}:}\quad It follows immediately from Propositions~\ref{prop0003} and \ref{prop0004}.  \qquad\raisebox{-0.5mm}{%
\rule{1.5mm}{4mm}}
\end{document}